\documentclass[12pt]{amsart}
\usepackage{xypic}
\usepackage{amssymb}
\usepackage{amsthm}
\usepackage{bm}
\usepackage{latexsym}
\usepackage{amsmath}
\usepackage{eufrak}
\usepackage{mathrsfs}
\usepackage{amscd}
\usepackage[all]{xy}
\usepackage[usenames]{color}
\usepackage[dvips]{graphicx}
\usepackage{color}
\usepackage{amscd}
\theoremstyle{plain}
\usepackage[usenames,dvipsnames]{pstricks}
\usepackage{epsfig}
\usepackage{pst-grad} 
\usepackage{pst-plot} 
\usepackage{longtable}

\newtheorem{theorem}{Theorem}[section]
\newtheorem{proposition}[theorem]{Proposition}
\newtheorem{lemma}[theorem]{Lemma}

\theoremstyle{definition}

\newcommand{\appsection}[1]{\let\oldthesection\thesection
\renewcommand{\thesection}{Appendix \oldthesection}
\section{#1}\let\thesection\oldthesection}

\theoremstyle{remark}

\def\D{{\mathbb{D}}}

\def\Z{{\mathbb{Z}}}
\def\Q{{\mathbb{Q}}}
\def\C{{\mathbb{C}}}
\def\P{{\mathbb{P}}}

\def\O{{\mathcal{O}}}

\def\X{{\mathcal{X}}}
\def\F{{\mathcal{F}}}

\begin{document}
\bibliographystyle{amsplain}
\title[Construcci\'on]{KSBA surfaces with elliptic quotient singularities, $\pi_1=1$, $p_g=0$, and $K^2=1,2$}
\author{\textrm{Ari\'e Stern and Giancarlo Urz\'ua}}

\email{stern@math.umass.edu}
\email{urzua@mat.puc.cl}

\maketitle

\begin{abstract}
Among log canonical surface singularities, the ones which have a rational homology disk smoothing are the cyclic quotient singularities $\frac{1}{n^2}(1,na-1)$ with gcd$(a,n)=1$, and three distinguished elliptic quotient singularities. We show the existence of smoothable KSBA normal surfaces with $\pi_1=1$, $p_g=0$, and $K^2=1,2$ for each of these three singularities. We also give a list of new (and old) normal surface singularities in smoothable KSBA surfaces for invariants $\pi_1=1$, $p_g=0$, and $K^2=1,2,3,4$.
\end{abstract}


\section{Introduction} \label{s-1}

An \textit{elliptic quotient singularity} is a normal two dimensional singularity which has discrepancy $(-1)$, and its canonical covering (index one cover) is a simple elliptic singularity. Although there are four possible groups to quotient a simple elliptic singularity, mainly $\Z/2\Z$, $\Z/3\Z$, $\Z/4\Z$, and $\Z/6\Z$, there are infinitely many such singularities. The list can be found in \cite[\S4]{GI10}. Among them, there are only three which admit a smoothing whose Milnor fiber has second Betti number equals to zero, i.e. a rational homology disk smoothing (c.f. \cite{Wahl81}). This smoothing is a $\Q$-Gorenstein smoothing over a smooth analytic curve germ; see \cite{Wahl13} for a general discussion. In fact these three elliptic quotient singularities and \textit{Wahl singularities} (i.e. cyclic quotient singularities $\frac{1}{n^2}(1,na-1)$ with gcd$(n,a)=1$) form the complete list of log canonical singularities which have a rational homology disk smoothing.

Let $(P \in X)$ be one of the three elliptic quotient singularities above. Then there is minimal resolution $\pi \colon \tilde{X} \to X$ such that the exceptional divisor consists of $4$ smooth rational curves $E_1$, $E_2$, $E_3$, and $F$. The curves $E_i$ are disjoint, each meets the central curve $F$ transversally at one point, and  $$[-E_1^2,-E_2^2,-E_3^2;-F^2] = [3,3,3;4], [4,2,4;3], \text{or} \, [2,3,6;2]$$ which correspond to $\Z/3\Z$, $\Z/4\Z$, and $\Z/6\Z$ respectively. Since these singularities are determined by their exceptional divisors, we refer to them using the symbol $[-E_1^2,-E_2^2,-E_3^2;-F^2]$. The set of these three singularities will be denoted by $\Q Eq$.

Let $X$ be a normal projective surface with either one Wahl singularity or one $\Q Eq$ singularity. Assume $X$ has no local-to-global obstructions to deform. Then the surface $X$ determines a codimension one component of the Koll\'ar--Shepherd-Barron--Alexeev (KSBA) boundary of the moduli space of surfaces of general type with fixed topological invariants $\chi(\O_X)$ and $K_X^2$; c.f. \cite{H11}. In the pioneering work \cite{LP07}, Lee and Park show the existence of such surfaces (see \cite{Urz13}) for invariants $p_g=0$ (and so $\chi=1$), trivial topological fundamental group, and $K^2=1,2$. Using a similar strategy, in this paper we prove the existence of such surfaces for all the $\Q Eq$ singularities and same invariants.

In \S \ref{s0}, we explain the Lee-Park method \cite{LP07} adapted to our situation. In \S \ref{s1} and \S \ref{s2}, we show the existence of the above mentioned KSBA surfaces with $\Q Eq$ singularities for $K^2=1$ and $K^2=2$ respectively. In \S \ref{s3}, we list the known Wahl singularities appearing for invariants $p_g=0$, $\pi_1=1$, and $K^2=1,2,3,4$ (the allowed $K^2$ for surfaces with no-local-to-global obstructions). In that list we show many new singularities. In particular we achieve the highest known indices for $K^2=2$ ($n=58$), and for $K^2=3$ ($n=123$). Also, we identify for each singularity the minimal smooth model of the general surface of the corresponding KSBA divisor, using the explicit birational geometry in \cite{HTU13} (see also \cite{Urz13}). The majority of them are rational surfaces, for explicitness one can use  \cite{Urz13b}.

\subsection*{Notation}

We use Kodaira's notation for singular fibers of elliptic fibrations. A canonical divisor of a normal surface $S$ is denoted by $K_S$. The strict transform of a curve under a birational map is denoted by the same letter. A SNC divisor in a smooth surface is a nodal divisor formed by nonsingular curves. We denote de dual of $\F$ by $\F^{\vee}$. We write $\pi_1(A)$ for the topological fundamental group of $A$ .

\subsection*{Acknowledgements}

We thank H. Park and D. Shin for providing us the list \cite{PSlist}. This is part of the master's thesis \cite{S13} of Ari\'e Stern at the Pontificia Universidad Cat\'olica de Chile. Both authors were supported by a FONDECYT Inicio grant funded by the Chilean Government (11110047).

\section{The Lee-Park method} \label{s0}

In this section we explain the Lee-Park method in \cite{LP07} slightly modified for $\Q Eq$ singularities. The steps below will be explicitly shown in each of the examples of the upcoming sections.

\textbf{Pencil of cubics.} We start with a suitable pencil of cubics in $\P^2$ which produces an elliptic fibration with sections. By blowing ups this fibration, we construct the exceptional divisors of Wahl and $\Q Eq$ singularities. Singular fibers and sections of this elliptic fibration, and maybe some other special curves, give us curves to begin the construction of the exceptional divisors. Many such exceptional configurations can be constructed from a given elliptic fibration. The point is that we need a ``equilibrium" on some data to construct the KSBA surface $X$ we want. This surface $X$ will be the contraction of the Wahl and/or $\Q Eq$ configurations. Notice that the configurations are exceptional divisors of rational singularities, and so by Artin's criterion \cite[III\S3]{BHPV04} they can be contracted to a normal projective surface $X$.
\bigskip

\textbf{No local-to-global obstructions.}
Given a normal projective surface $X$, we say that it has \textit{no-local-to-global obstructions} to deform if any deformation of each of the singularities of $X$ can be glued together to a global deformation of $X$. The obstruction relies on $H^2(X,{\Omega_X^1}^{\vee})$.

\begin{lemma}
Let $Z$ be a nonsingular projective surface. Let $W_1,\ldots,W_r$ be disjoint exceptional divisors in $Z$ of either Wahl or $\Q Eq$ singularities. Assume that $H^2\big(Z,{\Omega_{Z}^1(\text{log}\,(\sum_{i=1}^r W_i))}^{\vee} \big)=0$. Let $f \colon Z \to X$ be the contraction of $W_1,\ldots,W_r$. Then $X$ has no-local-to-global obstructions to deform.  \label{obstr=0}
\end{lemma}

\begin{proof}
The key is to prove $R^1f_*({\Omega_{Z}^1(\text{log}\,(\sum_{i=1}^r W_i))}^{\vee})=0$. For this, we only need that the singularities are rational and taut, to then apply the argument in \cite[Lemma 1]{LP07}. Note that $\Q Eq$ are rational and taut. See also \cite[\S 8]{Wahl11}. After that, it follows through the same argument as in \cite[Theorem 2]{LP07}.
\end{proof}

With the following lemmas one can prove the vanishing of the above cohomology in many cases. Their proofs can be found in \cite{LP07} and \cite{PSU13}.

\begin{lemma}
Let $g \colon Y \to \P^1$ be an elliptic fibration with at least one section. Assume $Y$ has two singular fibers $F_1$ and $F_2$ of type $I_n$ and $I_m$ with $n,m\ge 1$. Let $\pi \colon Y' \to Y$ be the blow-up of $Y$ in a node of $F_1$ and in a node of $F_2$. Then $H^2 \big(Y', {\Omega_{Y'}^1(\text{log}\,(F_1+F_2))}^{\vee}  \big)=0$. This is also true if we consider only one singular fiber. \label{initialzero}
\end{lemma}

\begin{lemma}
Let $Y$ be a nonsingular projective surface, and let $D$ be a $SNC$ divisor. Let $\pi \colon Y' \to Y$ be the blow-up of $Y$ at some point, and let $E$ be the corresponding $(-1)$-curve. Then, $H^2 \big(Y,{\Omega_Y^1(\text{log}\,D)}^{\vee} \big)=H^2 \big(Y',{\Omega_{Y'}^1(\text{log}\,(D+E))}^{\vee} \big)$.

Furthermore, if $G$ is a $(-1)$-curve on $Y$ such that $D+G$ is a $SNC$ divisor, then $H^2 \big(Y,{\Omega_Y^1(\text{log}\,D)}^{\vee} \big)=H^2 \big(Y,{\Omega_{Y}^1(\text{log}\,(D+G))}^{\vee} \big).$ \label{adddelete}
\end{lemma}

\begin{lemma}
Let $Y$ be a nonsingular projective surface, and let $D$ be a $SNC$ divisor. Assume that there exists configurations of curves $\{D_1,\ldots ,D_r\}$ which correspond to disjoint exceptional divisors of $ADE$ singularities. Assume that for $1\le i\le r$ we have $D \cap D_i=\emptyset$. Then $H^2 \big(Y,{\Omega_Y^1(\text{log}\,D)}^{\vee} \big)=H^2 \big(Y,{\Omega_{Y}^1(\text{log}\,(D+\sum_{i=1}^r D_i))}^{\vee} \big)$. \label{ADE}
\end{lemma}

Let $X$ be a normal projective surface with no-local-to-global obstructions to deform, and only Wahl and $\Q Eq$ singularities. Then, there are \textit{partial $\Q$-Gorenstein smoothings} $X \subset \X \to 0 \in \D$ of $X$ over an analytic smooth curve germ $\D$ for any subset of singularities of $X$. This is, the deformation $\X \to \D$ is locally on the singularities of $X$ either trivial (preserving the singularity) or a rational homology disk smoothing.
\bigskip

\textbf{Numerical invariants.} Assume $X$ to be a rational surface, as it will be in our examples. Let $X \subset \X \to 0 \in \D$ be any smoothing of $X$. Then by \cite[\S 3]{GS83}, the first Betti number of $X_t$ is constant for all $t$, where $X_t$ is the fiber at $t$. This implies that the irregularity $q(X_t)=\text{dim}_{\C} H^1(\O_{X_t})=0$. Since $\chi(X_t)$ is independent of $t$, we also have $p_g(X_t)=\text{dim}_{\C} H^2(\O_{X_t})=0$.

When $X \subset \X \to 0 \in \D$ is a partial $\Q$-Gorenstein smoothing, then $K_{X_t}^2=K_X^2$ for any $t$. We use the following known fact to compute it.

\begin{proposition}
Let $X$ be a normal projective surface with only Wahl and $\Q Eq$ singularities $Q_1,\ldots, Q_n$. Let $Z\to X$ be the minimal resolution, and let $l_i$ be the number of exceptional curves over $Q_i$. Then, $K_X^2=K_Z^2 + \sum_{i=1}^n l_i.$ \label{K2}
\end{proposition}
\bigskip

\textbf{Positivity of $K$.} Each of our examples $X$ will be built from a contraction $Z \to X$. We then manage to write the pullback in $Z$ of $K_X$ with positive rational coefficients. After that, we intersect it with each of the curves in its support, and verify that these intersection numbers are nonnegative. This implies that $K_X$ is nef, and so $K_{X_t}$ is nef, where $X \subset \X \to 0 \in \D$ is any partial $\Q$-Gorenstein smoothing of $X$. Given that $K_X$ is nef, we use the Nakai-Moishezon criteria to prove ampleness. We determine precisely which are the curves whose intersection with $K_X$ is zero, and then contract them (if necessary) to obtain the canonical model. It turns out that, in our examples, the pullback of $K_X$ under the contraction $Z \to X$ can be written with a $\Q$-effective support which contains all the components of a fiber of $Z \to \P^1$. Thus, we only need to look at components of fibers in $Z \to \P^1$, that is a simple verification. Notice that ampleness of $K_X$ implies ampleness of $K_{X_t}$ for any $t$.
\bigskip

\textbf{Fundamental group.} We use the same method as in \cite{LP07} to compute the fundamental group of $X_t$ from the data in $X$. The difference will be the local fundamental group of the $\Q$Eq singularities. Below we compute these groups using \cite{Mum61}.

Let $(P\in X)$ be a $\Q$Eq singularity, and let $\tilde{X}$ be its minimal resolution. We are going to compute the fundamental group of the complement in $\tilde{X}$ of the exceptional divisor $E_1+E_2+E_3+F$ (see \S \ref{s-1} for notation). Let $\alpha_i$ be a loop around the curve $E_i$, and let $\gamma$ be a loop around the central curve $F$. By \cite[p.12]{Mum61}, we have that the fundamental group is generated by these loops subject to the relations: $$ \alpha_i \gamma= \gamma \alpha_i \, , 1=\gamma \alpha_1^{E_1^2} \, , 1=\gamma \alpha_2^{E_2^2} \, , 1=\gamma \alpha_3^{E_3^2} \, ,1=\alpha_1 \alpha_2 \alpha_3 \gamma^{F^2}.$$

From these relations, for the singularity $[4,2,4;3]$, we get that $\gamma=\alpha_1^4=\alpha_3^4$ and $\alpha_2=\alpha_1^{-1}\alpha_3^{11}$. Hence $$\pi_1 \big( \tilde{X}\setminus [4,2,4;3] \big) =\langle \alpha_1,\alpha_3 \, | \, \alpha_1^4=\alpha_3^4 \rangle.$$

Analogously, for the other two $\Q Eq$  singularities we get that $$\pi_1 \big( \tilde{X}\setminus [3,3,3;4] \big) =\langle \alpha_1,\alpha_3 \, | \, \alpha_1^3=\alpha_3^3 \rangle$$ and $\pi_1 \big( \tilde{X}\setminus [2,3,6;2] \big) =\langle \alpha_1,\alpha_3 \, | \, \alpha_1^2=\alpha_3^6 \rangle.$

\section{$K^2=1$} \label{s1}

\subsection{$[4,2,4;3]$} \label{[4,2,4;3]}

Let $L_1,\ldots ,L_6$ be lines in general position in $\P^2$. Consider the pencil $$\Gamma_{\lambda, \mu}=\{\lambda L_1L_2L_3+\mu L_4L_5L_6=0\}$$
with $[\lambda: \mu] \in \P^1$, and let $Y \to \P^1$ be the elliptic fibration obtained by blowing up $\P^2$ at the base points. Note that there are two $I_3$ singular fibers in $Y\to \P^1$, which consists of the strict transforms of $L_1,L_2,L_3$ and $L_4,L_5,L_6$. There are also six nodal singular fibers. Let $Z \to Y$ be the blow-up on $10$ points of $Y$ as shown in the picture below. Relevant curves are the sections $E_1, E_2, E_3$, the chosen nodal fiber $F$, and the exceptional curves $G_1,\ldots, G_{10}$ of $Z \to Y$, whose subindices follow the order of the blow-ups.

\begin{figure}[htbp]
\scalebox{0.8} 
{
\begin{pspicture}(0,-2.1690094)(16.51125,2.1690094)
\psline[linewidth=0.04cm](9.644037,1.1519568)(8.345312,-0.5290094)
\psline[linewidth=0.04cm](9.420293,1.1519568)(10.511049,-1.0340924)
\psline[linewidth=0.04cm](8.485312,-0.86900944)(10.622922,-0.8519216)
\psline[linewidth=0.04cm,linestyle=dashed,dash=0.16cm 0.16cm](15.466428,1.0919567)(14.431607,-1.0940924)
\psline[linewidth=0.04cm](15.242682,1.0919567)(16.333439,-1.0940924)
\psline[linewidth=0.04cm](14.347702,-0.9119216)(16.445312,-0.9119216)
\usefont{T1}{ptm}{m}{n}
\rput(14.830937,0.39599058){\scriptsize $L_6$}
\usefont{T1}{ptm}{m}{n}
\rput(16.110937,-0.14400941){\scriptsize $L_5$}
\usefont{T1}{ptm}{m}{n}
\rput(15.110937,-0.74400944){\scriptsize $L_4$}
\usefont{T1}{ptm}{m}{n}
\rput(9.110937,0.8359906){\scriptsize $L_3$}
\usefont{T1}{ptm}{m}{n}
\rput(9.890938,0.7359906){\scriptsize $L_2$}
\usefont{T1}{ptm}{m}{n}
\rput(9.290937,-0.68400943){\scriptsize $L_1$}
\usefont{T1}{ptm}{m}{n}
\rput(12.410937,0.87599057){\scriptsize $F$}
\usefont{T1}{ptm}{m}{n}
\rput(11.150937,1.8559905){\scriptsize $E_3$}
\usefont{T1}{ptm}{m}{n}
\rput(13.570937,-0.18400942){\scriptsize $E_2$}
\usefont{T1}{ptm}{m}{n}
\rput(13.570937,-1.5240095){\scriptsize $E_1$}
\psline[linewidth=0.04cm](1.2040375,1.2119567)(0.16921687,-0.9740924)
\psline[linewidth=0.04cm](0.9802925,1.2119567)(2.0710495,-0.9740924)
\psline[linewidth=0.04cm](0.0853125,-0.7919216)(2.182922,-0.7919216)
\psline[linewidth=0.04cm](6.1264277,1.2119567)(5.091607,-0.9740924)
\psline[linewidth=0.04cm](5.9026823,1.2119567)(6.993439,-0.9740924)
\psline[linewidth=0.04cm](5.0077024,-0.7919216)(7.105312,-0.7919216)
\psbezier[linewidth=0.04](3.3855512,2.070762)(3.0779018,-0.29745814)(4.1686587,0.32712734)(3.805073,0.76954204)(3.4414876,1.2119567)(3.4974236,-1.4425315)(3.55336,-1.5726535)
\psline[linewidth=0.04cm](1.4837188,-0.21938495)(5.623001,-0.19336055)
\psbezier[linewidth=0.04](1.2879418,-0.6878241)(1.2879418,-1.7287998)(6.713758,-1.338434)(5.9586186,-0.66179967)
\psbezier[linewidth=0.04](0.86842,0.17098098)(0.50483435,0.24905416)(0.39296186,1.5242496)(1.1760694,1.7324446)(1.9591769,1.9406399)(6.2383,1.6543715)(6.517981,1.2119567)(6.7976623,0.76954204)(6.517981,0.6133957)(6.210332,0.24905416)
\usefont{T1}{ptm}{m}{n}
\rput(5.3909373,0.33599058){\scriptsize $L_6$}
\usefont{T1}{ptm}{m}{n}
\rput(6.8109374,0.07599058){\scriptsize $L_5$}
\usefont{T1}{ptm}{m}{n}
\rput(6.2109375,-0.6040094){\scriptsize $L_4$}
\usefont{T1}{ptm}{m}{n}
\rput(0.3309375,0.03599058){\scriptsize $L_3$}
\usefont{T1}{ptm}{m}{n}
\rput(1.6709375,0.37599057){\scriptsize $L_2$}
\usefont{T1}{ptm}{m}{n}
\rput(0.9909375,-0.62400943){\scriptsize $L_1$}
\usefont{T1}{ptm}{m}{n}
\rput(3.5509374,1.0359906){\scriptsize $F$}
\usefont{T1}{ptm}{m}{n}
\rput(2.4509375,1.5959905){\scriptsize $E_3$}
\usefont{T1}{ptm}{m}{n}
\rput(2.6309376,-0.024009418){\scriptsize $E_2$}
\usefont{T1}{ptm}{m}{n}
\rput(2.8509376,-1.1040094){\scriptsize $E_1$}
\usefont{T1}{ptm}{m}{n}
\rput(7.5867186,-0.019009417){$\longleftarrow$}
\psline[linewidth=0.04cm](12.565312,2.1109905)(12.585313,-1.5690094)
\psline[linewidth=0.04cm](12.485312,-0.049009416)(13.205313,0.7709906)
\psline[linewidth=0.04cm,linestyle=dashed,dash=0.16cm 0.16cm](12.425312,1.2709906)(13.245313,0.63099056)
\psline[linewidth=0.04cm](9.845312,0.13099058)(10.685312,0.13099058)
\psline[linewidth=0.04cm](10.505313,-0.3290094)(12.325313,-0.34900942)
\psline[linewidth=0.04cm](12.765312,-0.3690094)(14.965313,-0.34900942)
\psline[linewidth=0.04cm,linestyle=dashed,dash=0.16cm 0.16cm](12.105312,-0.28900942)(12.745313,-0.98900944)
\psline[linewidth=0.04cm,linestyle=dashed,dash=0.16cm 0.16cm](10.565312,0.27099058)(10.565312,-0.46900943)
\psbezier[linewidth=0.04](9.465313,-1.3290094)(12.825313,-2.1490095)(14.725312,-1.5652595)(15.625313,-1.4490094)
\psbezier[linewidth=0.04](8.325313,0.5309906)(8.725312,1.0309906)(8.886963,1.7529718)(9.985312,1.9509906)(11.083662,2.1490095)(14.845312,1.8909906)(15.645312,1.5509906)(16.445312,1.2109905)(15.705313,0.43099058)(15.565312,0.19099058)
\psline[linewidth=0.04cm,linestyle=dashed,dash=0.16cm 0.16cm](9.645312,-0.7290094)(9.705313,-1.4890094)
\psline[linewidth=0.04cm,linestyle=dashed,dash=0.16cm 0.16cm](15.4453125,-0.7890094)(15.485312,-1.5290095)
\psline[linewidth=0.04cm,linestyle=dashed,dash=0.16cm 0.16cm](12.685312,-1.2690095)(12.005313,-1.8090094)
\psline[linewidth=0.04cm,linestyle=dashed,dash=0.16cm 0.16cm](8.325313,0.6509906)(9.005313,0.05099058)
\psbezier[linewidth=0.04,linestyle=dashed,dash=0.16cm 0.16cm](8.265312,-0.3090094)(8.665313,-0.4290094)(8.905313,-0.6290094)(8.605312,-1.0490094)
\usefont{T1}{ptm}{m}{n}
\rput(15.840938,-1.1640095){\scriptsize $G_{10}$}
\usefont{T1}{ptm}{m}{n}
\rput(12.170938,-0.7240094){\scriptsize $G_9$}
\usefont{T1}{ptm}{m}{n}
\rput(12.050938,-1.4440094){\scriptsize $G_8$}
\usefont{T1}{ptm}{m}{n}
\rput(9.910937,-1.1440095){\scriptsize $G_7$}
\usefont{T1}{ptm}{m}{n}
\rput(8.850938,-0.4240094){\scriptsize $G_6$}
\usefont{T1}{ptm}{m}{n}
\rput(8.370937,0.27599058){\scriptsize $G_5$}
\usefont{T1}{ptm}{m}{n}
\rput(10.810938,-0.06400942){\scriptsize $G_4$}
\usefont{T1}{ptm}{m}{n}
\rput(10.250937,0.2959906){\scriptsize $G_3$}
\usefont{T1}{ptm}{m}{n}
\rput(13.030937,1.0559906){\scriptsize $G_2$}
\usefont{T1}{ptm}{m}{n}
\rput(13.090938,0.27599058){\scriptsize $G_1$}
\usefont{T1}{ptm}{m}{n}
\rput(12.430938,0.45599058){\tiny -7}
\usefont{T1}{ptm}{m}{n}
\rput(11.632031,-0.22400942){\tiny -4}
\usefont{T1}{ptm}{m}{n}
\rput(15.686719,-0.7840094){\tiny -3}
\usefont{T1}{ptm}{m}{n}
\rput(15.09125,-0.024009418){\tiny -2}
\usefont{T1}{ptm}{m}{n}
\rput(9.152031,0.27599058){\tiny -4}
\usefont{T1}{ptm}{m}{n}
\rput(11.252031,-1.5240095){\tiny -4}
\usefont{T1}{ptm}{m}{n}
\rput(9.872031,-0.7040094){\tiny -4}
\usefont{T1}{ptm}{m}{n}
\rput(9.866718,-0.16400942){\tiny -3}
\usefont{T1}{ptm}{m}{n}
\rput(15.63125,-0.08400942){\tiny -2}
\usefont{T1}{ptm}{m}{n}
\rput(10.23125,-0.0040094173){\tiny -2}
\usefont{T1}{ptm}{m}{n}
\rput(13.05125,1.7959906){\tiny -2}
\usefont{T1}{ptm}{m}{n}
\rput(12.79125,0.49599057){\tiny -2}
\end{pspicture}
}
\end{figure}

Let $Z \to X$ be the contraction of the following three Wahl configurations, and one $[4,2,4;3]$: $$E_1=[4], \ E_2=[4], \ G_1,F,E_3,L_5,L_4=[2,7,2,2,3],$$   $$L_1,G_3,L_3;L_2=[4,2,4;3].$$ Then we have, by applying several times the lemmas in Section \ref{s0}, that $X$ has no-local-to-global obstructions to deform (see \cite{LP07,Urz13}).

We compute via Proposition \ref{K2} that $K_X^2=-10+1+1+4+5=1$, and so a $\Q$-Gorenstein smoothing of $X$ is a nonsingular projective surface with $q=p_g=0$ and $K^2=1$.

We have that $$K_Z \equiv \frac{-F}{2}+\frac{-L_1}{2}+\frac{-L_2}{2}+\frac{-L_3}{2}+\frac{G_2}{2}+\frac{G_3}{2}+\frac{3G_4}{2}+\frac{G_5}{2}+
\frac{G_7}{2}+\frac{G_8}{2}+\frac{G_9}{2}+G_{10}$$ and, by adding the discrepancies, the pullback $f^*(K_X)$ is numerically equivalent to $$\frac{7}{18}F+\frac{1}{4}L_1+\frac{1}{2}L_2+\frac{1}{4}L_3+\frac{4}{9}G_1+\frac{1}{2}G_2+G_3+\frac{3}{2}G_4+\frac{1}{2}G_5+$$ $$\frac{1}{2}G_7+\frac{1}{2}G_8+\frac{1}{2}G_9+G_{10}+\frac{1}{2}E_1+\frac{1}{2}E_2+\frac{7}{9}E_3+\frac{5}{9}L_4+\frac{6}{9}L_5.$$

As explained in Section \ref{s0}, we now intersect $f^*(K_X)$ with all the curves in its $\Q$-effective support. One can verify that these numbers are nonnegative. In addition, its support contains all the components of a fiber of $Z \to \P^1$, mainly $G_1, G_2, F, G_8, G_9$. Therefore, a curve $\Gamma$ in $X$ with $K_X \cdot \Gamma=0$ has strict transform $\Gamma$ in $Z$ with $\Gamma \cdot f^*(K_X)=0$, and it is part of a fiber. One can verify this cannot happen, so by the Nakai-Moishezon criteria, the canonical class $K_X$ of $X$ is ample.

In this way, we have shown that $X$ is a KSBA smoothable surface, and also the existence of KSBA surfaces $X'$ with one $[4,2,4;3]$ singularity, $K^2=1$, and $q=p_g=0$. Using the technique explained in \cite{Urz13}, one can prove that $X'$ is rational.

For the fundamental group we do the following. It is known that given a singularity $\frac{1}{m}(1,q)$, the fundamental group of its link is $\Z /m\Z$. If $E_1+\ldots+E_s$ represents the exceptional divisor of its minimal resolution, which is a chain of $\P^1$'s in that order, then the group is generated by a loop around $E_1$ (or by a loop around $E_s$) \cite{Mum61}. Another known fact is that given two disjoint exceptional configurations $W_1$ and $W_2$, and a $\P^1$ which intersects one component of the exceptional divisor $W_1$ and one component of the exceptional divisor of $W_2$ at one point each, then we get that loops around these components are homotopic in the fundamental group of the complement, exactly as used in \cite{LP07}.

In this example, using the curve $G_{10}$, we get that loops around $E_1$ and loops around $L_4$ are homotopic. Since the order of the fundamental groups of theirs links are coprime, we conclude that both of these loops are trivial in the complement. Then by using the curve $G_9$, we get that a loop around $E_2$ is trivial, and then using the curves $G_4, G_5$ and $G_7$, we obtain that loops around $G_3, L_3$ and $L_1$ are trivial. Finally, we know (by the Mumford's relations explained in Section \ref{s0})  that $1=\gamma \alpha^{-4}$ where $\gamma$ is a loop around $L_2$ and $\alpha$ is a loop around $L_1$. Since $\alpha$ is trivial, we conclude that $\gamma$ is trivial. Hence the fundamental group of the complement of the exceptional configurations in $Z$ is trivial. This is enough to conclude that a $\Q$-Gorenstein smoothing of $X$ is simply connected; cf. \cite{LP07}.

\subsection{$[3,3,3;4]$} \label{[3,3,3;4]}

Let $L_1, L_2, L_3, L$ be lines in general position in $\P^2$. Let $C$ be a general conic. Consider the pencil $$\Gamma_{\lambda,\mu}=\{\lambda L_1L_2L_3 +\mu LC=0\}$$
with $[\lambda:\mu]\in \P^1$, and let $Y \to \P^1$ be the elliptic fibration obtained by blowing up $\P^2$ at the base points. Note that there is one $I_3$ and one $I_2$ singular fibers, corresponding to the proper transforms of $L_1,L_2,L_3$ and $L,C$ respectively. We also get other $7$ nodal singular fibers. Let $Z \to Y$ be the blow-up on $9$ points of $Y$ as shown in the picture below. Relevant curves are the sections $E_1, E_2$, the chosen nodal fiber $F$, and the exceptional curves $G_1,\ldots, G_{9}$ of $Z \to Y$, whose subindices follow the order of the blow-ups.

The computations below are as we did in the previous example, so we omit details.

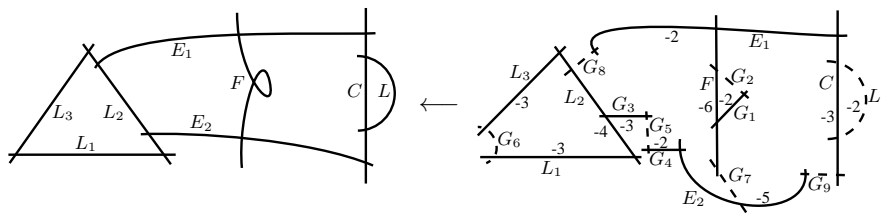
\begin{figure}[htbp]
\scalebox{0.8} 
{
\begin{pspicture}(0,-1.87)(14.645938,1.87)
\psline[linewidth=0.04cm](1.42,1.19)(0.02,-0.79)
\psline[linewidth=0.04cm](0.0,-0.65)(2.78,-0.65)
\psline[linewidth=0.04cm](1.24,1.19)(2.7,-0.81)
\psline[linewidth=0.04cm](5.94,1.79)(5.94,-1.13)
\rput{-89.31165}(5.3603477,6.1651363){\psarc[linewidth=0.04](5.8,0.37){0.62}{0.0}{180.0}}
\psbezier[linewidth=0.04](3.8,1.71)(3.8,0.61)(4.48,-0.07)(4.34,0.57)(4.2,1.21)(3.8,0.01)(3.88,-0.87)
\psbezier[linewidth=0.04](2.2,-0.31)(3.08,-0.27)(5.0,-0.39)(6.06,-0.83)
\psbezier[linewidth=0.04](1.44,0.79)(1.82,1.51)(5.18,1.33)(6.1,1.37)
\psline[linewidth=0.04cm](9.26,1.15)(7.8,-0.33)
\psline[linewidth=0.04cm](7.84,-0.69)(10.52,-0.69)
\psline[linewidth=0.04cm](9.08,1.15)(10.5,-0.81)
\psline[linewidth=0.04cm](13.78,1.75)(13.78,-1.17)
\rput{-89.31165}(13.185914,13.859389){\psarc[linewidth=0.04,linestyle=dashed,dash=0.16cm 0.16cm](13.606409,0.257051){0.6467318}{0.0}{180.0}}
\psbezier[linewidth=0.04](11.16,-0.39)(11.04,-1.67)(13.4,-1.85)(13.24,-0.89)
\psbezier[linewidth=0.04](9.8,1.01)(8.96,1.85)(13.02,1.29)(13.94,1.33)
\psbezier[linewidth=0.04,linestyle=dashed,dash=0.16cm 0.16cm](7.8,-0.19)(8.06,-0.17)(8.28,-0.65)(7.94,-0.79)
\psline[linewidth=0.04cm,linestyle=dashed,dash=0.16cm 0.16cm](9.8,1.13)(9.24,0.67)
\psline[linewidth=0.04cm](9.82,-0.01)(10.7,-0.01)
\psline[linewidth=0.04cm](10.52,-0.57)(11.26,-0.57)
\psline[linewidth=0.04cm,linestyle=dashed,dash=0.16cm 0.16cm](10.6,0.07)(10.64,-0.63)
\psline[linewidth=0.04cm,linestyle=dashed,dash=0.16cm 0.16cm](13.16,-0.97)(13.9,-0.99)
\psline[linewidth=0.04cm](11.76,1.67)(11.78,-1.01)
\psline[linewidth=0.04cm](11.68,-0.21)(12.28,0.41)
\psline[linewidth=0.04cm,linestyle=dashed,dash=0.16cm 0.16cm](11.66,0.85)(12.3,0.29)
\psline[linewidth=0.04cm,linestyle=dashed,dash=0.16cm 0.16cm](11.66,-0.73)(12.26,-1.61)
\usefont{T1}{ptm}{m}{n}
\rput(6.245625,0.415){\scriptsize $L$}
\usefont{T1}{ptm}{m}{n}
\rput(0.905625,0.055){\scriptsize $L_3$}
\usefont{T1}{ptm}{m}{n}
\rput(1.745625,0.055){\scriptsize $L_2$}
\usefont{T1}{ptm}{m}{n}
\rput(1.285625,-0.465){\scriptsize $L_1$}
\usefont{T1}{ptm}{m}{n}
\rput(5.755625,0.415){\scriptsize $C$}
\usefont{T1}{ptm}{m}{n}
\rput(3.805625,0.555){\scriptsize $F$}
\usefont{T1}{ptm}{m}{n}
\rput(2.885625,1.075){\scriptsize $E_1$}
\usefont{T1}{ptm}{m}{n}
\rput(3.205625,-0.125){\scriptsize $E_2$}
\usefont{T1}{ptm}{m}{n}
\rput(7.161406,0.16){$\longleftarrow$}
\usefont{T1}{ptm}{m}{n}
\rput(9.025625,-0.885){\scriptsize $L_1$}
\usefont{T1}{ptm}{m}{n}
\rput(9.405625,0.295){\scriptsize $L_2$}
\usefont{T1}{ptm}{m}{n}
\rput(8.505625,0.715){\scriptsize $L_3$}
\usefont{T1}{ptm}{m}{n}
\rput(12.485625,1.235){\scriptsize $E_1$}
\usefont{T1}{ptm}{m}{n}
\rput(11.385625,-1.385){\scriptsize $E_2$}
\usefont{T1}{ptm}{m}{n}
\rput(11.605625,0.535){\scriptsize $F$}
\usefont{T1}{ptm}{m}{n}
\rput(13.595625,0.555){\scriptsize $C$}
\usefont{T1}{ptm}{m}{n}
\rput(14.385625,0.335){\scriptsize $L$}
\usefont{T1}{ptm}{m}{n}
\rput(13.485625,-1.145){\scriptsize $G_9$}
\usefont{T1}{ptm}{m}{n}
\rput(9.745625,0.795){\scriptsize $G_8$}
\usefont{T1}{ptm}{m}{n}
\rput(12.165625,-1.005){\scriptsize $G_7$}
\usefont{T1}{ptm}{m}{n}
\rput(8.325625,-0.405){\scriptsize $G_6$}
\usefont{T1}{ptm}{m}{n}
\rput(10.845625,-0.165){\scriptsize $G_5$}
\usefont{T1}{ptm}{m}{n}
\rput(10.845625,-0.725){\scriptsize $G_4$}
\usefont{T1}{ptm}{m}{n}
\rput(10.225625,0.175){\scriptsize $G_3$}
\usefont{T1}{ptm}{m}{n}
\rput(12.205625,0.655){\scriptsize $G_2$}
\usefont{T1}{ptm}{m}{n}
\rput(12.245625,0.035){\scriptsize $G_1$}
\usefont{T1}{ptm}{m}{n}
\rput(11.58625,0.155){\tiny -6}
\usefont{T1}{ptm}{m}{n}
\rput(12.543282,-1.365){\tiny -5}
\usefont{T1}{ptm}{m}{n}
\rput(9.846719,-0.265){\tiny -4}
\usefont{T1}{ptm}{m}{n}
\rput(13.621407,-0.0050){\tiny -3}
\usefont{T1}{ptm}{m}{n}
\rput(10.845938,-0.445){\tiny -2}
\usefont{T1}{ptm}{m}{n}
\rput(10.281406,-0.165){\tiny -3}
\usefont{T1}{ptm}{m}{n}
\rput(9.141406,-0.545){\tiny -3}
\usefont{T1}{ptm}{m}{n}
\rput(8.541407,0.215){\tiny -3}
\usefont{T1}{ptm}{m}{n}
\rput(11.025937,1.315){\tiny -2}
\usefont{T1}{ptm}{m}{n}
\rput(11.925938,0.275){\tiny -2}
\usefont{T1}{ptm}{m}{n}
\rput(14.045938,0.155){\tiny -2}
\end{pspicture}
}

\end{figure}

$E_2,G_4=[5,2], \ G_1,F,E_1,C=[2,6,2,3], \ L_1,L_3,G_3;L_2=[3,3,3;4]$

Let $f \colon Z \to X$ be the contraction of these three configurations. Then we have $K^2_X=-9+2+4+4=1$, and $$K_Z\equiv \frac{-F}{2}+\frac{-L_1}{2}+\frac{-L_2}{2}+\frac{-L_3}{2}+\frac{G_2}{2}+\frac{G_3}{2}+\frac{3}{2}G_4+3G_5+\frac{G_7}{2}+\frac{G_8}{2}+G_9$$

\begin{eqnarray*}
f^*(K_X) & \equiv & \frac{5}{14}F+\frac{1}{6}L_1+\frac{1}{2}L_2+\frac{1}{6}L_3+\frac{3}{7}G_1+\frac{1}{2}G_2+\frac{7}{6}G_3+\frac{11}{6}G_4\\
         && +3G_5+\frac{1}{2}G_7+\frac{1}{2}G_8+G_9+\frac{5}{7}E_1+\frac{2}{3}E_2+\frac{4}{7}C
\end{eqnarray*}

As before one proves that $X$ has no-local-to-global obstructions to deform, and that $X$ is a KSBA surface. It produces KSBA surfaces with the singularity $[3,3,3;4]$. These surfaces are rational via \cite{HTU13} (see \cite{Urz13}).

Using the curve $G_9$ we get that loops around $E_2$ and loops around $C$ are homotopic. Since the orders of the fundamental groups of its links are coprime, we get that both loops are trivial. Then by using $G_5$ and $G_8$ we get that loops around $G_3$ and $L_2$ are trivial. From the construction of the pencil, we see that there exists a section passing through $L_3$ and $F$. By using this section we get that loops around these curves are homotopic, and since loops around $F$ are trivial, then we have that loops around $L_3$ are also trivial. Finally, using $G_6$ we have that loops around $L_1$ are trivial, and we conclude that the fundamental group of the complement of these configurations in $Z$ is trivial. As before, this is enough to conclude that a $\Q$-Gorenstein smoothing of $X$ is simply connected; cf. \cite{LP07}.

\subsection{$[2,3,6;2]$} \label{[2,3,6;2]}

Let $C_1,C_2,L_1$ and $L_2$ be two conics and two lines respectively in general position in $\P^2$. Consider the pencil
$$\Gamma_{\lambda, \mu}=\{\lambda C_1L_1+\mu C_2L_2=0\}$$
with $[\lambda : \mu] \in \P^1$, and let $Y \to \P^1$ be the elliptic fibration obtained by blowing up $\P^2$ at the base points. Note that there are two $I_2$ singular fibers in $Y\to \P^1$, which consist of the strict transform of $C_1,L_1$ and $C_2,L_2$. There are also eight nodal singular fibers. Let $M$ be a line through one point in $L_1\cap C_2 $ and the node of one of the nodal singular fibers. Then after blowing up the base points, $M$ becomes a double section for the elliptic fibration $Y \to \P^1$. Let $Z\to Y$ be the blow-up on $6$ points of $Y$ as shown in the picture below. Relevant curves are the section $E_1$, the chosen nodal fibers $F_1,F_2$ and the exceptional curves $G_1, \ldots, G_6$ of $Z \to Y$, whose subindices follow the order of the blow-ups.

The computations below are as we did in the first example, so we omit details.

\begin{figure}[htbp]
\scalebox{0.9} 
{
\begin{pspicture}(0,-2.2804883)(13.107758,2.3004882)
\psline[linewidth=0.04cm](0.9218401,2.0995116)(0.9218401,-1.7004883)
\psline[linewidth=0.04cm](5.52184,2.0995116)(5.52184,-1.7004883)
\psbezier[linewidth=0.04](2.26184,2.0195117)(2.10184,0.35951173)(2.78184,-0.58048826)(2.76184,0.2595117)(2.7418401,1.0995117)(2.10184,-0.10048828)(2.26184,-1.6804882)
\psbezier[linewidth=0.04](3.90184,1.9995117)(3.7418401,0.33951172)(4.42184,-0.6004883)(4.40184,0.23951171)(4.38184,1.0795118)(3.7418401,-0.12048828)(3.90184,-1.7004883)
\rput{-90.13889}(0.5241758,1.0019301){\psarc[linewidth=0.04](0.76184005,0.23951171){0.76}{-8.391876}{189.60121}}
\rput{-90.13889}(5.175375,5.601868){\psarc[linewidth=0.04](5.38184,0.21951172){0.76}{-11.171042}{194.17514}}
\psbezier[linewidth=0.04](0.8418401,1.6395117)(1.46184,1.5195117)(4.3218403,1.4395118)(4.70184,1.4395118)(5.08184,1.4395118)(5.44184,1.4795117)(5.3218403,0.85951173)
\psbezier[linewidth=0.04](1.6218401,0.8395117)(1.22184,0.8195117)(0.9818401,0.13951172)(1.58184,0.13951172)(2.1818402,0.13951172)(2.54184,0.2595117)(2.6818402,0.23951171)
\psbezier[linewidth=0.04](2.88184,0.2595117)(3.6818402,0.27951172)(4.28184,-0.42048827)(4.18184,-0.7604883)(4.08184,-1.1004883)(3.4618402,-1.5804883)(3.9618402,-1.9204882)(4.46184,-2.2604883)(4.75525,-0.8425009)(4.92184,-0.52048826)(5.08843,-0.19847564)(5.3218403,0.119511716)(6.3418403,0.07951172)
\usefont{T1}{ptm}{m}{n}
\rput(1.1774162,-1.6754882){\scriptsize $L_1$}
\usefont{T1}{ptm}{m}{n}
\rput(5.2374163,-1.5354882){\scriptsize $L_2$}
\usefont{T1}{ptm}{m}{n}
\rput(1.3774163,-0.5354883){\scriptsize $C_1$}
\usefont{T1}{ptm}{m}{n}
\rput(6.1174164,0.8845117){\scriptsize $C_2$}
\usefont{T1}{ptm}{m}{n}
\rput(3.2274163,-0.03548828){\scriptsize $M$}
\usefont{T1}{ptm}{m}{n}
\rput(2.9674163,0.6245117){\scriptsize $F_1$}
\usefont{T1}{ptm}{m}{n}
\rput(4.627416,0.48451173){\scriptsize $F_2$}
\psline[linewidth=0.04cm](7.50184,2.0995116)(7.50184,-1.7004883)
\psline[linewidth=0.04cm](12.10184,2.0995116)(12.10184,-1.7004883)
\rput{-90.13889}(7.1201262,7.5819106){\psarc[linewidth=0.04,linestyle=dashed,dash=0.16cm 0.16cm](7.3418403,0.23951171){0.76}{-8.391876}{189.60121}}
\rput{-90.13889}(11.771325,12.181849){\psarc[linewidth=0.04,linestyle=dashed,dash=0.16cm 0.16cm](11.96184,0.21951172){0.76}{-11.171042}{194.17514}}
\psbezier[linewidth=0.04](8.32184,0.8995117)(7.92184,0.6995117)(7.452275,0.09951172)(7.92184,0.09951172)(8.391405,0.09951172)(8.40184,0.059511717)(8.80184,0.09951172)
\psbezier[linewidth=0.04](9.00184,0.119511716)(10.04184,0.15951172)(9.86184,-0.22048828)(9.92184,-0.70048827)(9.98184,-1.1804882)(10.44184,-1.7804883)(10.72184,-1.5004883)(11.00184,-1.2204883)(11.33525,-0.8425009)(11.50184,-0.52048826)(11.66843,-0.19847564)(11.90184,0.119511716)(12.92184,0.07951172)
\usefont{T1}{ptm}{m}{n}
\rput(7.7574162,-1.6754882){\scriptsize $L_1$}
\usefont{T1}{ptm}{m}{n}
\rput(11.817416,-1.5354882){\scriptsize $L_2$}
\usefont{T1}{ptm}{m}{n}
\rput(7.957416,-0.5354883){\scriptsize $C_1$}
\usefont{T1}{ptm}{m}{n}
\rput(12.697416,0.8845117){\scriptsize $C_2$}
\usefont{T1}{ptm}{m}{n}
\rput(8.447416,0.22451171){\scriptsize $M$}
\usefont{T1}{ptm}{m}{n}
\rput(9.167417,1.2845117){\scriptsize $F_1$}
\usefont{T1}{ptm}{m}{n}
\rput(10.7474165,0.92451173){\scriptsize $F_2$}
\psline[linewidth=0.04cm](8.92184,2.0995116)(8.92184,-1.7004883)
\psline[linewidth=0.04cm](10.52184,1.6995118)(10.52184,-1.7004883)
\psline[linewidth=0.04cm](9.70184,0.71951175)(9.70184,-0.7204883)
\psline[linewidth=0.04cm,linestyle=dashed,dash=0.16cm 0.16cm](9.84184,-0.90048826)(10.64184,-0.90048826)
\psline[linewidth=0.04cm,linestyle=dashed,dash=0.16cm 0.16cm](8.76184,0.6395117)(9.80184,0.6395117)
\psline[linewidth=0.04cm,linestyle=dashed,dash=0.16cm 0.16cm](8.76184,-0.64048827)(9.80184,-0.64048827)
\psbezier[linewidth=0.04](7.3618402,1.9795117)(8.20184,1.9595118)(10.36184,1.9395118)(11.1218405,1.8595117)(11.88184,1.7795117)(11.84184,1.9395118)(11.92184,0.8195117)
\psline[linewidth=0.04cm,linestyle=dashed,dash=0.16cm 0.16cm](9.88184,1.9995117)(10.64184,1.2195117)
\rput{-88.549866}(9.939404,10.573228){\psarc[linewidth=0.04,linestyle=dashed,dash=0.16cm 0.16cm](10.39184,0.18951172){0.45}{0.0}{180.0}}
\usefont{T1}{ptm}{m}{n}
\rput(11.087417,0.22451171){\scriptsize $G_1$}
\usefont{T1}{ptm}{m}{n}
\rput(9.947416,0.2645117){\scriptsize $G_2$}
\usefont{T1}{ptm}{m}{n}
\rput(9.2474165,0.8245117){\scriptsize $G_3$}
\usefont{T1}{ptm}{m}{n}
\rput(9.227416,-0.8154883){\scriptsize $G_4$}
\usefont{T1}{ptm}{m}{n}
\rput(9.927416,1.5645118){\scriptsize $G_5$}
\usefont{T1}{ptm}{m}{n}
\rput(10.207417,-0.7354883){\scriptsize $G_6$}
\usefont{T1}{ptm}{m}{n}
\rput(10.368139,0.16451173){\tiny -6}
\usefont{T1}{ptm}{m}{n}
\rput(8.788139,-0.23548828){\tiny -6}
\usefont{T1}{ptm}{m}{n}
\rput(9.543251,-0.23548828){\tiny -3}
\usefont{T1}{ptm}{m}{n}
\rput(11.207817,-0.6954883){\tiny -2}
\usefont{T1}{ptm}{m}{n}
\rput(11.947817,0.30451173){\tiny -2}
\usefont{T1}{ptm}{m}{n}
\rput(9.547816,2.0445118){\tiny -2}
\usefont{T1}{ptm}{m}{n}
\rput(12.507816,0.4445117){\tiny -2}
\usefont{T1}{ptm}{m}{n}
\rput(7.3678164,0.30451173){\tiny -2}
\usefont{T1}{ptm}{m}{n}
\rput(7.8278165,1.0245117){\tiny -2}
\usefont{T1}{ptm}{m}{n}
\rput(6.713295,0.28451172){$\longleftarrow$}
\usefont{T1}{ptm}{m}{n}
\rput(2.9974163,1.6645117){\scriptsize $E_1$}
\usefont{T1}{ptm}{m}{n}
\rput(8.197416,2.1445117){\scriptsize $E_1$}
\end{pspicture}
}
\end{figure}

\noindent
$F_1,E_1,L_1=[6,2,2]$, $L_2,G_2,F_2;M=[2,3,6;2]$, $K_X^2=-6+3+4=1$, $$K_Z \equiv -\frac{1}{2}F_1-\frac{1}{2}F_2+\frac{1}{2}G_3+\frac{1}{2}G_4+\frac{1}{2}G_5+\frac{1}{2}G_6$$

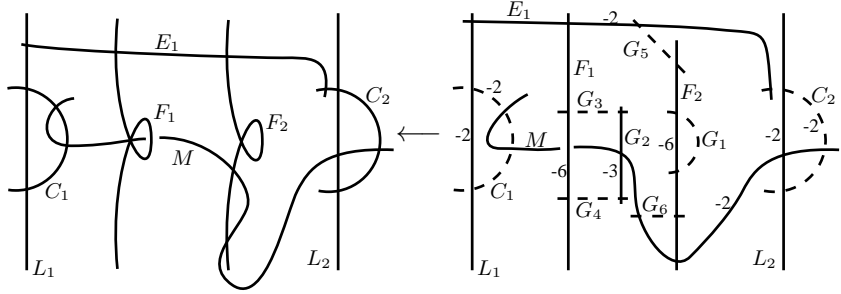
\begin{eqnarray*}
f^*(K_X) &\equiv & \frac{1}{4}F_1+\frac{1}{3}F_2+\frac{1}{2}E_1+\frac{1}{4}L_1+\frac{1}{2}L_2+M\\
         & & + \frac{2}{3}G_2+\frac{1}{2}G_3+\frac{1}{2}G_4+\frac{1}{2}G_5+\frac{1}{2}G_6
\end{eqnarray*}

As before one proves that $X$ has no-local-to-global obstructions to deform, and it is a KSBA surface. It produces KSBA surfaces with the singularity $[2,3,6;2]$. These surfaces are rational via \cite{Urz13}.

Let $\alpha_1,\alpha_2,\alpha_3,\gamma$ and $\beta$ be loops around $L_2,G_2,F_2,M$ and $F_1$ respectively. Then using the curve $G_1$ we get that $\alpha_3^2=1$, and by the Mumford's relations explained in Section \ref{s0} we get that $\alpha_1^2=\alpha_2^3=\alpha_3^6=1$. Using the curve $G_3$ we have that $\beta$ and $\alpha_2$ are homotopic, then $\alpha_2^{16}=1$ since $\beta$ has order 16. But $\alpha_2^3=1$ and $\alpha_2^{16}=1$ implies that $\alpha_2=1$. Then $\beta=1$ and using the curves $G_5$ and $G_6$ we get that $\alpha_3=1$ and $\gamma=1$. Finally, using Mumford's relations again we have that $1=\alpha_1 \alpha_2 \alpha_3 \gamma^{-2}$ so $\alpha_1=1$. Hence the fundamental group of the complement of the exceptional configurations in $Z$ is trivial.

\section{$K^2=2$} \label{s2}

\subsection{$[4,2,4;3]$} \label{[4,2,4;3]'}

Let $L_1,L_2,L_3$ and $L$ be general lines in $\P^2$, and let $M$ be a general line passing through $L_2 \cap L_3$. Consider the pencil $$\Gamma_{\lambda,\mu}=\{\lambda L_1L_2L_3+\mu L^2M=0\}$$ with $[\lambda:\mu]\in \P^1$. Let $Y \to \P^1$ be the elliptic fibration obtained by blowing up $\P^2$ at the base points. Note that there is one $I_4$ singular fiber corresponding to the triangle $L_1L_2L_3$. There is also one $I_0^*$ singular fiber which consists of the strict transform of $L,M$ and the first exceptional curves of the blow-ups at the points of intersection between $L$ and $L_1,L_2,L_3$. Let $N$ be a line passing through the intersection of $L$ with $L_3$, and by the node of one of the $I_1$ fibers of the pencil. Then after blowing up the base points, $N$ becomes a double section of the elliptic fibration. Let $Z\to Y$ be the blow-up on $13$ points of $Y$ as shown in the picture below. Relevant curves are the sections $E_2$ and $E_7$, the chosen nodal fiber $F$, and the exceptional curves $G_1, \ldots, G_{13}$ of $Z \to Y$, whose subindices follow the order of the blow-ups.

The computations below are as we did in the previous section, so we omit details, except in the case of no-local-to-global obstructions. For that, Lemma \ref{initialzero} can be easily adapted (see e.g. \cite[\S 4]{PSU13}) to have $F_1$ nodal and $F_2$ a simple normal crossings singular fiber. In our case $F_2$ is the $I_0^*$ fiber (and reduced).

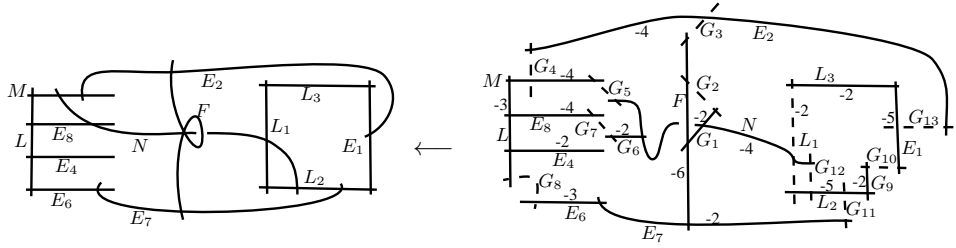
\begin{figure}[htbp]
\scalebox{0.8} 
{
\begin{pspicture}(0,-2.0301561)(15.93125,2.0101562)
\psline[linewidth=0.04cm](8.407454,0.7864247)(8.425312,-1.0698438)
\psline[linewidth=0.04cm](8.348311,0.7050185)(9.9841585,0.7050185)
\psline[linewidth=0.04cm](8.583526,-1.3118557)(10.025312,-1.3298438)
\psline[linewidth=0.04cm](8.323526,0.12652917)(9.9841585,0.12652917)
\psline[linewidth=0.04cm](8.323526,-0.46266326)(9.9841585,-0.46266326)
\psbezier[linewidth=0.04](8.665313,1.2101562)(9.165313,1.2101562)(10.525312,1.7301563)(11.194905,1.7635291)(11.864497,1.7969018)(15.125313,1.4501562)(15.4453125,1.0501562)(15.765312,0.65015626)(15.665313,0.39015624)(15.665313,-0.20984375)
\psbezier[linewidth=0.04](11.487609,-0.03398576)(11.870376,-0.03398576)(12.825047,-0.33322713)(13.025312,-0.46984375)(13.225578,-0.6064604)(13.005313,-0.6898438)(13.485312,-0.66984373)
\psline[linewidth=0.04cm](0.48483393,0.58780307)(0.46281767,-1.2186674)
\psline[linewidth=0.04cm](0.39676887,0.45721483)(1.8498425,0.45721483)
\psline[linewidth=0.04cm](0.37475258,-1.1098439)(1.8498425,-1.1098439)
\psline[linewidth=0.04cm](0.37475258,-0.02160865)(1.8498425,-0.02160865)
\psline[linewidth=0.04cm](0.37475258,-0.5657262)(1.8498425,-0.5657262)
\psbezier[linewidth=0.04](2.8185585,1.0448618)(2.620412,0.15250897)(3.3689651,-0.8921968)(3.3029163,-0.15219687)(3.2368674,0.58780307)(2.6763794,-0.34925565)(2.9853125,-1.6098437)
\psbezier[linewidth=0.04](1.3214521,0.37015602)(1.1853125,1.0701562)(1.8921843,0.7960383)(2.9253125,0.85015625)(3.9584408,0.90427417)(6.0043807,1.1348469)(6.3253126,0.79015625)(6.6462445,0.44546562)(6.4053125,-0.06984375)(6.0053124,-0.22984375)
\psbezier[linewidth=0.04](3.3853126,-0.16984375)(3.7853124,-0.16984375)(4.235665,-0.17690273)(4.4853125,-0.34984374)(4.73496,-0.52278477)(4.8853126,-0.64984375)(4.8853126,-1.1898438)
\psline[linewidth=0.04cm](11.352514,1.5036473)(11.385312,-1.7898438)
\psline[linewidth=0.04cm](11.262452,-0.46712187)(11.870376,0.2475527)
\psline[linewidth=0.04cm,linestyle=dashed,dash=0.16cm 0.16cm](11.262452,0.7889728)(11.925312,0.11015625)
\psline[linewidth=0.04cm,linestyle=dashed,dash=0.16cm 0.16cm](8.765312,1.3301562)(8.765312,0.43015626)
\psline[linewidth=0.04cm,linestyle=dashed,dash=0.16cm 0.16cm](9.705313,0.23015624)(10.185312,-0.30984375)
\psline[linewidth=0.04cm](9.998933,-0.22783192)(10.685312,-0.22984375)
\psline[linewidth=0.04cm,linestyle=dashed,dash=0.16cm 0.16cm](11.262452,1.2870793)(11.885312,1.9901563)
\psline[linewidth=0.04cm,linestyle=dashed,dash=0.16cm 0.16cm](15.785313,-0.06984375)(14.627457,-0.071955666)
\usefont{T1}{ptm}{m}{n}
\rput(2.2709374,-0.36484376){\scriptsize $N$}
\usefont{T1}{ptm}{m}{n}
\rput(0.2309375,0.53515625){\scriptsize $M$}
\usefont{T1}{ptm}{m}{n}
\rput(0.2909375,-0.28484374){\scriptsize $L$}
\usefont{T1}{ptm}{m}{n}
\rput(5.1109376,0.47515625){\scriptsize $L_3$}
\usefont{T1}{ptm}{m}{n}
\rput(5.1709375,-0.90484375){\scriptsize $L_2$}
\usefont{T1}{ptm}{m}{n}
\rput(4.6109376,-0.06484375){\scriptsize $L_1$}
\usefont{T1}{ptm}{m}{n}
\rput(3.3109374,0.21515626){\scriptsize $F$}
\usefont{T1}{ptm}{m}{n}
\rput(0.9909375,-0.18484375){\scriptsize $E_8$}
\usefont{T1}{ptm}{m}{n}
\rput(0.9709375,-1.2848438){\scriptsize $E_6$}
\usefont{T1}{ptm}{m}{n}
\rput(1.0509375,-0.72484374){\scriptsize $E_4$}
\usefont{T1}{ptm}{m}{n}
\rput(3.4709375,0.6951563){\scriptsize $E_2$}
\usefont{T1}{ptm}{m}{n}
\rput(14.260938,-1.4448438){\scriptsize $G_{11}$}
\usefont{T1}{ptm}{m}{n}
\rput(14.600938,-0.54484373){\scriptsize $G_{10}$}
\usefont{T1}{ptm}{m}{n}
\rput(14.610937,-1.0048437){\scriptsize $G_9$}
\usefont{T1}{ptm}{m}{n}
\rput(9.090938,-0.96484375){\scriptsize $G_8$}
\usefont{T1}{ptm}{m}{n}
\rput(9.690937,-0.10484375){\scriptsize $G_7$}
\usefont{T1}{ptm}{m}{n}
\rput(10.390938,-0.40484375){\scriptsize $G_6$}
\usefont{T1}{ptm}{m}{n}
\rput(10.250937,0.5751563){\scriptsize $G_5$}
\usefont{T1}{ptm}{m}{n}
\rput(9.010938,0.9351562){\scriptsize $G_4$}
\usefont{T1}{ptm}{m}{n}
\rput(11.770938,1.5351562){\scriptsize $G_3$}
\usefont{T1}{ptm}{m}{n}
\rput(11.7109375,0.65515625){\scriptsize $G_2$}
\usefont{T1}{ptm}{m}{n}
\rput(11.7109375,-0.28484374){\scriptsize $G_1$}
\usefont{T1}{ptm}{m}{n}
\rput(13.650937,0.8151562){\scriptsize $L_3$}
\usefont{T1}{ptm}{m}{n}
\rput(13.690937,-1.3248438){\scriptsize $L_2$}
\usefont{T1}{ptm}{m}{n}
\rput(13.390938,-0.28484374){\scriptsize $L_1$}
\usefont{T1}{ptm}{m}{n}
\rput(11.230938,0.37515625){\scriptsize $F$}
\usefont{T1}{ptm}{m}{n}
\rput(12.630938,1.4951563){\scriptsize $E_2$}
\usefont{T1}{ptm}{m}{n}
\rput(9.310938,-0.6248438){\scriptsize $E_4$}
\usefont{T1}{ptm}{m}{n}
\rput(8.930938,-0.02484375){\scriptsize $E_8$}
\usefont{T1}{ptm}{m}{n}
\rput(9.510938,-1.5048437){\scriptsize $E_6$}
\usefont{T1}{ptm}{m}{n}
\rput(8.290937,-0.18484375){\scriptsize $L$}
\usefont{T1}{ptm}{m}{n}
\rput(8.130938,0.71515626){\scriptsize $M$}
\usefont{T1}{ptm}{m}{n}
\rput(12.390938,-0.02484375){\scriptsize $N$}
\usefont{T1}{ptm}{m}{n}
\rput(11.211562,-0.82484376){\tiny -6}
\usefont{T1}{ptm}{m}{n}
\rput(8.266719,0.31515625){\tiny -3}
\usefont{T1}{ptm}{m}{n}
\rput(9.426719,-1.1848438){\tiny -3}
\usefont{T1}{ptm}{m}{n}
\rput(9.29125,-0.30484375){\tiny -2}
\usefont{T1}{ptm}{m}{n}
\rput(13.688594,-1.0448438){\tiny -5}
\usefont{T1}{ptm}{m}{n}
\rput(9.372031,0.8151562){\tiny -4}
\usefont{T1}{ptm}{m}{n}
\rput(9.372031,0.25515625){\tiny -4}
\usefont{T1}{ptm}{m}{n}
\rput(12.352032,-0.42484376){\tiny -4}
\usefont{T1}{ptm}{m}{n}
\rput(10.612031,1.5351562){\tiny -4}
\usefont{T1}{ptm}{m}{n}
\rput(10.29125,-0.10484375){\tiny -2}
\usefont{T1}{ptm}{m}{n}
\rput(14.01125,0.47515625){\tiny -2}
\usefont{T1}{ptm}{m}{n}
\rput(14.23125,-0.96484375){\tiny -2}
\usefont{T1}{ptm}{m}{n}
\rput(11.79125,-1.5848438){\tiny -2}
\usefont{T1}{ptm}{m}{n}
\rput(11.59125,0.11515625){\tiny -2}
\usefont{T1}{ptm}{m}{n}
\rput(7.146719,-0.45984375){$\longleftarrow$}
\psline[linewidth=0.04cm](4.3653126,0.71015626)(4.3828177,-1.1786673)
\psline[linewidth=0.04cm](6.084834,0.7078031)(6.1053123,-1.1898438)
\psline[linewidth=0.04cm](4.2453127,0.6301563)(6.1853123,0.61015624)
\psline[linewidth=0.04cm](4.2453127,-1.0698438)(6.2053127,-1.1098437)
\psbezier[linewidth=0.04](0.8653125,0.5701563)(1.3253125,-0.44984376)(2.8053124,-0.12984376)(3.2053125,-0.16984375)
\psbezier[linewidth=0.04](1.6253124,-0.9898437)(0.9653125,-1.7098438)(5.9853125,-1.5698438)(5.6053123,-1.0098437)
\usefont{T1}{ptm}{m}{n}
\rput(2.2909374,-1.5848438){\scriptsize $E_7$}
\usefont{T1}{ptm}{m}{n}
\rput(5.8109374,-0.38484374){\scriptsize $E_1$}
\psbezier[linewidth=0.04,linestyle=dashed,dash=0.16cm 0.16cm](8.305312,-0.94984376)(8.87762,-0.82984376)(8.9453125,-0.84984374)(8.829927,-1.4298438)
\psline[linewidth=0.04cm,linestyle=dashed,dash=0.16cm 0.16cm](13.105312,0.73015624)(13.145312,-1.3498437)
\psline[linewidth=0.04cm](14.825313,0.6901562)(14.885312,-0.84984374)
\psline[linewidth=0.04cm](13.005313,0.6301563)(14.885312,0.61015624)
\psline[linewidth=0.04cm](12.985312,-1.1498437)(14.485312,-1.1698438)
\psbezier[linewidth=0.04](10.045313,0.37015626)(10.4453125,0.37015626)(10.565312,0.39015624)(10.585313,0.07015625)(10.605312,-0.24984375)(10.685312,-0.6898438)(10.825313,-0.58984375)(10.965313,-0.48984376)(10.885312,0.05015625)(11.225312,-0.00984375)
\psline[linewidth=0.04cm,linestyle=dashed,dash=0.16cm 0.16cm](9.725312,0.79015625)(10.265312,0.29015625)
\psline[linewidth=0.04cm,linestyle=dashed,dash=0.16cm 0.16cm](13.405313,-0.50984377)(13.425312,-1.2898438)
\psline[linewidth=0.04cm,linestyle=dashed,dash=0.16cm 0.16cm](13.965313,-0.9898437)(14.005313,-1.7498437)
\psbezier[linewidth=0.04](9.885312,-1.2298437)(9.905313,-1.9898437)(13.285313,-1.5698438)(14.105312,-1.6298437)
\psline[linewidth=0.04cm](14.345312,-0.64984375)(14.365313,-1.2898438)
\psline[linewidth=0.04cm,linestyle=dashed,dash=0.16cm 0.16cm](14.245313,-0.72984374)(15.005313,-0.7498438)
\usefont{T1}{ptm}{m}{n}
\rput(13.740937,-0.7448437){\scriptsize $G_{12}$}
\usefont{T1}{ptm}{m}{n}
\rput(15.300938,0.09515625){\scriptsize $G_{13}$}
\usefont{T1}{ptm}{m}{n}
\rput(13.27125,0.21515626){\tiny -2}
\usefont{T1}{ptm}{m}{n}
\rput(14.708593,0.09515625){\tiny -5}
\usefont{T1}{ptm}{m}{n}
\rput(15.110937,-0.42484376){\scriptsize $E_1$}
\usefont{T1}{ptm}{m}{n}
\rput(10.790937,-1.8448437){\scriptsize $E_7$}
\end{pspicture}
}
\end{figure}

\noindent
$E_2=[4]$, $E_1,L_3=[5,2]$, $L_2,G_9=[5,2]$, $M,E_4,E_8;L=[4,2,4;3]$, $E_6,E_7,F,G_1,N,G_6=[3,2,6,2,4,2]$, $K_X^2=-13+1+2+2+6+4=2$,

\begin{eqnarray*}
K_Z &\equiv & \frac{-F}{2}+\frac{-L}{2}+\frac{-M}{2}+\frac{-E_4}{2}+\frac{-E_6}{2}+\frac{-E_8}{2}+\frac{1}{2}G_2+\frac{1}{2}G_3\\
    & &+\frac{1}{2}G_4+\frac{1}{2}G_5+\frac{1}{2}G_6+G_7+G_9+2G_{10}+G_{11}+G_{12}+G_{13}
\end{eqnarray*}

\begin{eqnarray*}
f^*(K_X) &\equiv & \frac{15}{34}F+\frac{1}{2}L+\frac{1}{4}M+\frac{3}{34}E_6+\frac{1}{4}E_8+\frac{15}{17}G_1+\frac{1}{2}G_2+\frac{1}{2}G_3\\
   & & +\frac{1}{2}G_4+\frac{1}{2}G_5+\frac{31}{34}G_6+G_7+\frac{4}{3}G_9+2G_{10}+G_{11}+G_{12}\\
   & & +G_{13}+\frac{2}{3}E_1+\frac{1}{2}E_2+\frac{2}{3}L_2+\frac{1}{3}L_3+\frac{13}{17}E_7+\frac{14}{17}N
\end{eqnarray*}

Therefore we have $\Q$-Gorenstein smoothable KSBA surfaces with one singularity $[4,2,4;3]$, $p_g=0$, and $K^2=2$. These surfaces are rational via \cite{Urz13}. We notice that $X$ itself is not KSBA since $G_{10}$ is a zero curve and the only one. Thus the corresponding KSBA surface is the contraction of $G_{10}$.

We prove that the fundamental group is trivial as before. Using the curve $G_{13}$ we get that loops around $E_2$, $E_1$ and $L_3$ are trivial. Then using $G_{10}$ we get that loops around $G_9$ and $L_2$ are trivial. Then using $G_4$ we have that loops around $M$ are trivial, and using Mumford's relations we get that loops around $L$ are trivial. Then using $G_8$ we get that loops around $E_6$ are trivial, so loops around $E_7,F,G_1,N,G_6$ are also trivial.  Then using $G_7$ we get that loops around $E_8$ are trivial and finally, using Mumford's relations once again we have that loops around $E_4$ are trivial. Hence the fundamental group of the complement of the exceptional configurations in $Z$ is trivial.

\subsection{$[3,3,3;4]$} \label{[3,3,3;4]'}

Let $L_1,L_2,L_3,L$ be general lines in $\P^2$, and let $C$ be a conic tangent to $L_1$ at $L_1 \cap L_2$ and general everywhere else. Consider the pencil
$$\Gamma_{\lambda,\mu}=\{\lambda L_1L_2L_3+\mu CL=0\}$$ with $[\lambda:\mu]\in \P^1$. Let $Y\to \P^1$ be the elliptic fibration obtained by blowing up $\P^2$ at the base points. Note that there is one $I_5$ singular fiber corresponding to the triangle $L_1L_2L_3$. There is also one $I_2$ singular fiber which consist of the proper transforms of $C$ and $L$. We also get four nodal singular fibers. Let $Z\to Y$ be the blow-up on $9$ points of $Y$ as shown in the picture below. Relevant curves are the sections $E_4$, $E_7$ and $E_9$, the chosen nodal fiber $F$, and the exceptional curves $G_1, \ldots, G_{9}$ of $Z \to Y$, whose subindices follow as always the order of the blow-ups. Again, the computations below are as we did in the previous section, so we omit details.

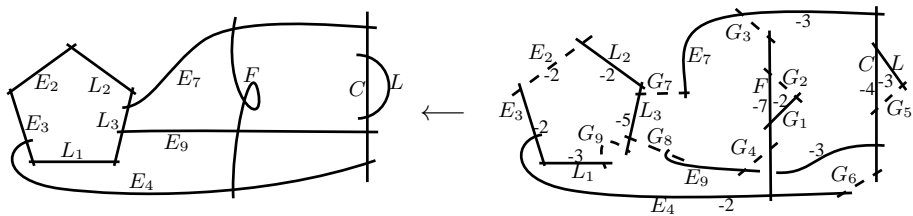
\begin{figure}[htbp]
\scalebox{0.9} 
{
\begin{pspicture}(0,-1.5901562)(13.745937,1.5701562)
\psline[linewidth=0.04cm](0.46,-0.7498438)(1.8,-0.7498438)
\psline[linewidth=0.04cm](0.18,0.33015624)(0.54,-0.8098438)
\psline[linewidth=0.04cm](2.0,0.35015625)(1.72,-0.8098438)
\psline[linewidth=0.04cm](0.14,0.25015625)(1.16,0.99015623)
\psline[linewidth=0.04cm](1.02,0.99015623)(2.04,0.25015625)
\psline[linewidth=0.04cm](5.46,1.4701562)(5.46,-1.0498438)
\rput{-88.63011}(4.8230014,5.6600285){\psarc[linewidth=0.04](5.31,0.36015624){0.47}{0.0}{180.0}}
\psbezier[linewidth=0.04](3.5142858,1.3901563)(3.2628572,0.23015624)(3.96,-0.32984376)(3.8495238,0.27015626)(3.7390475,0.8701562)(3.44,-0.48984376)(3.48,-1.2898438)
\psbezier[linewidth=0.04](1.84,0.05015625)(2.26,0.07015625)(2.5,0.65015626)(3.02,1.0301563)(3.54,1.4101562)(4.98,1.2701563)(5.6,1.2301563)
\psbezier[linewidth=0.04](1.76,-0.30984375)(2.34,-0.26984376)(4.9,-0.30984375)(5.62,-0.30984375)
\psbezier[linewidth=0.04](0.52,-0.42984375)(0.1,-0.48984376)(0.0,-1.0298438)(0.94,-1.1498437)(1.88,-1.2698437)(3.88,-1.3098438)(5.58,-0.72984374)
\usefont{T1}{ptm}{m}{n}
\rput(5.315625,0.33515626){\scriptsize $C$}
\usefont{T1}{ptm}{m}{n}
\rput(5.885625,0.43515626){\scriptsize $L$}
\usefont{T1}{ptm}{m}{n}
\rput(3.745625,0.53515625){\scriptsize $F$}
\usefont{T1}{ptm}{m}{n}
\rput(1.105625,-0.58484375){\scriptsize $L_1$}
\usefont{T1}{ptm}{m}{n}
\rput(1.625625,-0.14484376){\scriptsize $L_3$}
\usefont{T1}{ptm}{m}{n}
\rput(1.505625,0.37515625){\scriptsize $L_2$}
\usefont{T1}{ptm}{m}{n}
\rput(0.725625,0.41515625){\scriptsize $E_2$}
\usefont{T1}{ptm}{m}{n}
\rput(0.585625,-0.18484375){\scriptsize $E_3$}
\usefont{T1}{ptm}{m}{n}
\rput(2.125625,-1.0448438){\scriptsize $E_4$}
\usefont{T1}{ptm}{m}{n}
\rput(2.825625,0.47515625){\scriptsize $E_7$}
\usefont{T1}{ptm}{m}{n}
\rput(2.645625,-0.46484375){\scriptsize $E_9$}
\psline[linewidth=0.04cm](8.0,-0.76984376)(9.08,-0.76984376)
\psline[linewidth=0.04cm](7.7,0.41015625)(8.08,-0.8098438)
\psline[linewidth=0.04cm](9.52,0.43015626)(9.28,-0.64984375)
\psline[linewidth=0.04cm,linestyle=dashed,dash=0.16cm 0.16cm](7.6,0.23015624)(8.78,1.1101563)
\psline[linewidth=0.04cm](8.54,1.0701562)(9.56,0.33015624)
\psline[linewidth=0.04cm](12.98,1.5501562)(12.98,-1.0498438)
\psbezier[linewidth=0.04](10.16,0.19015625)(10.18,0.59015626)(9.94,1.1101563)(10.54,1.3301562)(11.14,1.5501562)(12.5,1.4501562)(13.1,1.4301562)
\psbezier[linewidth=0.04](9.92,-0.56984377)(9.62,-0.82984376)(10.66,-0.84984374)(11.26,-0.88984376)
\psbezier[linewidth=0.04](8.04,-0.34984374)(7.62,-0.40984374)(7.52,-1.0098437)(8.46,-1.1698438)(9.4,-1.3298438)(10.88,-1.2698437)(12.62,-1.2098438)
\usefont{T1}{ptm}{m}{n}
\rput(12.835625,0.6951563){\scriptsize $C$}
\usefont{T1}{ptm}{m}{n}
\rput(13.285625,0.71515626){\scriptsize $L$}
\usefont{T1}{ptm}{m}{n}
\rput(11.265625,0.37515625){\scriptsize $F$}
\usefont{T1}{ptm}{m}{n}
\rput(9.825625,-1.4048438){\scriptsize $E_4$}
\usefont{T1}{ptm}{m}{n}
\rput(10.365625,0.8151562){\scriptsize $E_7$}
\usefont{T1}{ptm}{m}{n}
\rput(10.325625,-0.98484373){\scriptsize $E_9$}
\usefont{T1}{ptm}{m}{n}
\rput(8.645625,-0.90484375){\scriptsize $L_1$}
\usefont{T1}{ptm}{m}{n}
\rput(9.645625,-0.00484375){\scriptsize $L_3$}
\usefont{T1}{ptm}{m}{n}
\rput(9.185625,0.85515624){\scriptsize $L_2$}
\usefont{T1}{ptm}{m}{n}
\rput(8.025625,0.85515624){\scriptsize $E_2$}
\usefont{T1}{ptm}{m}{n}
\rput(7.585625,-0.00484375){\scriptsize $E_3$}
\psline[linewidth=0.04cm](12.92,1.0101563)(13.42,0.31015626)
\psline[linewidth=0.04cm,linestyle=dashed,dash=0.16cm 0.16cm](13.42,0.43015626)(12.9,-0.04984375)
\psline[linewidth=0.04cm,linestyle=dashed,dash=0.16cm 0.16cm](13.06,-0.8698437)(12.4,-1.2698437)
\psline[linewidth=0.04cm,linestyle=dashed,dash=0.16cm 0.16cm](9.26,-0.36984375)(10.14,-0.72984374)
\psline[linewidth=0.04cm,linestyle=dashed,dash=0.16cm 0.16cm](9.42,0.25015625)(10.22,0.27015626)
\psbezier[linewidth=0.04,linestyle=dashed,dash=0.16cm 0.16cm](9.34,-0.60984373)(8.92,-0.26984376)(8.84,-0.52984375)(8.98,-0.84984374)
\psline[linewidth=0.04cm](11.4,1.1901562)(11.42,-1.3298438)
\psline[linewidth=0.04cm,linestyle=dashed,dash=0.16cm 0.16cm](10.9,1.5101563)(11.48,1.0101563)
\psline[linewidth=0.04cm,linestyle=dashed,dash=0.16cm 0.16cm](11.32,0.6301563)(11.88,0.13015625)
\psline[linewidth=0.04cm](11.86,0.27015626)(11.32,-0.26984376)
\psline[linewidth=0.04cm,linestyle=dashed,dash=0.16cm 0.16cm](10.94,-0.94984376)(11.52,-0.46984375)
\psbezier[linewidth=0.04](11.5,-0.90984374)(12.12,-0.94984376)(12.08,-0.44984376)(13.1,-0.50984377)
\usefont{T1}{ptm}{m}{n}
\rput(8.765625,-0.34484375){\scriptsize $G_9$}
\usefont{T1}{ptm}{m}{n}
\rput(9.785625,-0.38484374){\scriptsize $G_8$}
\usefont{T1}{ptm}{m}{n}
\rput(9.785625,0.43515626){\scriptsize $G_7$}
\usefont{T1}{ptm}{m}{n}
\rput(12.505625,-0.94484377){\scriptsize $G_6$}
\usefont{T1}{ptm}{m}{n}
\rput(13.325625,0.03515625){\scriptsize $G_5$}
\usefont{T1}{ptm}{m}{n}
\rput(11.025625,-0.54484373){\scriptsize $G_4$}
\usefont{T1}{ptm}{m}{n}
\rput(10.945625,1.1551563){\scriptsize $G_3$}
\usefont{T1}{ptm}{m}{n}
\rput(11.785625,0.49515626){\scriptsize $G_2$}
\usefont{T1}{ptm}{m}{n}
\rput(11.785625,-0.12484375){\scriptsize $G_1$}
\usefont{T1}{ptm}{m}{n}
\rput(11.265625,0.07515625){\tiny -7}
\usefont{T1}{ptm}{m}{n}
\rput(9.243281,-0.12484375){\tiny -5}
\usefont{T1}{ptm}{m}{n}
\rput(12.846719,0.33515626){\tiny -4}
\usefont{T1}{ptm}{m}{n}
\rput(13.121407,0.45515624){\tiny -3}
\usefont{T1}{ptm}{m}{n}
\rput(8.541407,-0.66484374){\tiny -3}
\usefont{T1}{ptm}{m}{n}
\rput(12.101406,-0.5648438){\tiny -3}
\usefont{T1}{ptm}{m}{n}
\rput(11.901406,1.3351562){\tiny -3}
\usefont{T1}{ptm}{m}{n}
\rput(8.185938,0.53515625){\tiny -2}
\usefont{T1}{ptm}{m}{n}
\rput(9.005938,0.53515625){\tiny -2}
\usefont{T1}{ptm}{m}{n}
\rput(8.025937,-0.22484376){\tiny -2}
\usefont{T1}{ptm}{m}{n}
\rput(10.765938,-1.3848437){\tiny -2}
\usefont{T1}{ptm}{m}{n}
\rput(11.565937,0.15515625){\tiny -2}
\usefont{T1}{ptm}{m}{n}
\rput(6.581406,-0.05984375){$\longleftarrow$}
\end{pspicture}
}
\end{figure}

\noindent
$L_3,L_2=[5,2]$, $G_1,F,E_4,E_3,L_1=[2,7,2,2,3]$,

\noindent
$L,E_7,E_9;C=[3,3,3;4]$, $K_X^2=-9+2+5+4=2$,

\begin{eqnarray*}
K_Z &\equiv & \frac{-F}{2}+\frac{-C}{2}+\frac{-L}{2}+\frac{1}{2}G_2+\frac{1}{2}G_3+\frac{1}{2}G_4\\
  & & +\frac{1}{2}G_6+G_7+G_8+G_9
\end{eqnarray*}

\begin{eqnarray*}
f^*(K_X) &\equiv & \frac{7}{18}F+\frac{1}{2}C+\frac{1}{6}L+\frac{4}{9}G_1+\frac{1}{2}G_2+\frac{1}{2}G_3\\
 & & +\frac{1}{2}G_4+\frac{1}{2}G_6+G_7+G_8+G_9+\frac{5}{9}L_1\\
 & & +\frac{1}{3}L_2+\frac{2}{3}L_3+\frac{6}{9}E_3+\frac{7}{9}E_4+\frac{2}{3}E_7+\frac{2}{3}E_9
\end{eqnarray*}

Therefore we have $\Q$-Gorenstein smoothable KSBA surfaces with one singularity $[3,3,3;4]$, $p_g=0$, and $K^2=2$. These surfaces are rational via \cite{Urz13}.

Let $\alpha_2$ be a loop around $E_7$. Using the curve $G_7$ we have that $\alpha_2$ is homotopic to a loop around $L_3$ which has order nine, then we get that $\alpha_2^9=1$. On the other hand, by Mumford's relations, we have that $\alpha_2^3=\gamma$, and then $1=\alpha_2^9=\gamma^3$. Using Mumford's relations again, we have that $\alpha_1^3=\gamma$ where $\alpha_1$ is a loop around $L$, and using the curve $G_5$ we get that $\alpha_1$ is homotopic to $\gamma$, then by combining these two facts we get that $\gamma$ is homotopic to $\gamma^3=1$, i.e. $\gamma=1$. From the construction of the elliptic fibration, one can see that if we take the line through $L_1\cap L_2$ and $L \cap L_3$, then after blowing up base points we get a section that intersects $E_2$, $F$, and $C$. By using this section we get that loops around $F$ are trivial since they are homotopic to $\gamma$. Then using the curves $G_2$ and $G_9$ have that loops around $G_1$ and $L_3$ are trivial, and so the fundamental group of the complement of the exceptional configurations in $Z$ is trivial. Therefore we are in the simply connected case again.

\subsection{$[2,3,6;2]$}

Let $L_1,L_2,L_3$ be general lines in $\P^2$. Let $L$ be a general line passing through $L_1\cap L_3$, and let $C$ be a conic which is tangent to $L_1$ in $L_1\cap L_2$ and general everywhere else. Consider the pencil $$\Gamma_{\lambda,\mu}=\{\lambda L_1L_2L_3+\mu CL=0\}$$ with $[\lambda:\mu]\in \P^1$. Let $Y\to \P^1$ be the elliptic fibration obtained by blowing up $\P^2$ at the base points. Note that there is one $I_6$ singular fiber on this fibration corresponding to the triangle $L_1L_2L_3$. There is also one $I_2$ singular fiber which consists of the proper transforms of $C$ and $L$. We also get four nodal singular fibers. Let $Z\to Y$ be the blow-up on $11$ points of $Y$ as shown in the picture below. Relevant curves are the sections $E_3$, $E_5$ and $E_7$, the chosen nodal fibers $F_1$ and $F_2$, and the exceptional curves $G_1, \ldots, G_{11}$ of $Z \to Y$, whose subindices follow as always the order of the blow-ups.

Again, the computations below are as we did in the previous section, so we omit details.

\begin{figure}[htbp]
\scalebox{0.8} 
{
\begin{pspicture}(0,-2.2201562)(15.445937,2.2001562)
\psline[linewidth=0.04cm](8.937635,0.7789219)(10.628832,0.7789219)
\psline[linewidth=0.04cm](9.19782,0.91453826)(8.547359,-0.306009)
\psline[linewidth=0.04cm,linestyle=dashed,dash=0.16cm 0.16cm](10.368649,0.91453826)(11.019109,-0.306009)
\psline[linewidth=0.04cm](8.547359,-0.17039263)(9.19782,-1.39094)
\psline[linewidth=0.04cm](11.019109,-0.17039263)(10.49874,-1.39094)
\psline[linewidth=0.04cm,linestyle=dashed,dash=0.16cm 0.16cm](8.937635,-1.2553235)(10.7589245,-1.2553235)
\psline[linewidth=0.04cm](10.68,-0.75984377)(11.694614,-0.75148827)
\psbezier[linewidth=0.04](8.54,-1.0998437)(7.98,-1.2960085)(8.742497,-1.8655972)(9.184811,-1.8927205)(9.627123,-1.9198438)(13.659978,-1.7571042)(14.232383,-1.8520355)
\psline[linewidth=0.04cm](14.518585,1.6061817)(14.518585,-1.6486111)
\rput{-89.926735}(14.275107,14.648171){\psarc[linewidth=0.04,linestyle=dashed,dash=0.16cm 0.16cm](14.471011,0.1773991){0.560153}{-11.053917}{190.84953}}
\psbezier[linewidth=0.04](10.056427,0.60262066)(10.108464,0.8060452)(9.920246,1.4110898)(10.6,1.5201563)(11.279754,1.6292228)(14.661687,2.0401542)(14.960898,1.4976887)(15.26011,0.9552232)(14.992935,0.7897439)(14.84,0.46015626)
\usefont{T1}{ptm}{m}{n}
\rput(10.905625,-1.0948437){\scriptsize $L_2$}
\usefont{T1}{ptm}{m}{n}
\rput(8.665625,0.42515624){\scriptsize $L_1$}
\usefont{T1}{ptm}{m}{n}
\rput(9.225625,-0.8748438){\scriptsize $E_2$}
\usefont{T1}{ptm}{m}{n}
\rput(10.145625,-2.0348437){\scriptsize $E_3$}
\usefont{T1}{ptm}{m}{n}
\rput(9.665625,0.96515626){\scriptsize $E_4$}
\usefont{T1}{ptm}{m}{n}
\rput(11.205625,1.8051562){\scriptsize $E_5$}
\usefont{T1}{ptm}{m}{n}
\rput(12.625625,-0.91484374){\scriptsize $E_7$}
\usefont{T1}{ptm}{m}{n}
\rput(14.355625,0.30515626){\scriptsize $C$}
\usefont{T1}{ptm}{m}{n}
\rput(11.685625,0.38515624){\scriptsize $F_1$}
\usefont{T1}{ptm}{m}{n}
\rput(12.965625,0.98515624){\scriptsize $F_2$}
\psline[linewidth=0.04cm,linestyle=dashed,dash=0.16cm 0.16cm](13.9721985,-1.8927205)(14.60965,-1.39094)
\psline[linewidth=0.04cm](11.94,1.8001562)(11.942762,-1.5943645)
\psline[linewidth=0.04cm](13.230674,1.524812)(13.243683,-1.540118)
\psline[linewidth=0.04cm,linestyle=dashed,dash=0.16cm 0.16cm](12.007809,-1.3631318)(11.513458,-1.9191588)
\psline[linewidth=0.04cm,linestyle=dashed,dash=0.16cm 0.16cm](13.288729,-1.4038167)(12.794379,-1.9055972)
\psline[linewidth=0.04cm,linestyle=dashed,dash=0.16cm 0.16cm](11.540449,-0.79998136)(12.073827,-0.31176245)
\psline[linewidth=0.04cm](12.034799,-0.7450499)(14.636641,-0.7450499)
\psline[linewidth=0.04cm,linestyle=dashed,dash=0.16cm 0.16cm](12.834379,1.7974144)(13.302711,1.2820723)
\psline[linewidth=0.04cm](11.877716,0.019470274)(12.5,0.0)
\psline[linewidth=0.04cm](12.463131,-0.021214636)(12.463131,0.6704288)
\psline[linewidth=0.04cm,linestyle=dashed,dash=0.16cm 0.16cm](12.521186,0.62974393)(11.8,0.64015627)
\rput{-92.10738}(13.392599,13.597094){\psarc[linewidth=0.04,linestyle=dashed,dash=0.16cm 0.16cm](13.249279,0.34412125){0.32430542}{327.089}{228.54854}}
\psline[linewidth=0.04cm](8.148148,-0.5494335)(8.499396,-1.1868305)
\psline[linewidth=0.04cm,linestyle=dashed,dash=0.16cm 0.16cm](8.14,-0.71984375)(8.76347,-0.35244736)
\usefont{T1}{ptm}{m}{n}
\rput(8.295625,-0.37484375){\scriptsize $G_{11}$}
\usefont{T1}{ptm}{m}{n}
\rput(8.095625,-0.95484376){\scriptsize $G_{10}$}
\usefont{T1}{ptm}{m}{n}
\rput(14.105625,-1.4948437){\scriptsize $G_9$}
\usefont{T1}{ptm}{m}{n}
\rput(12.805625,1.4451562){\scriptsize $G_8$}
\usefont{T1}{ptm}{m}{n}
\rput(12.785625,-1.5348438){\scriptsize $G_7$}
\usefont{T1}{ptm}{m}{n}
\rput(11.465625,-1.5348438){\scriptsize $G_6$}
\usefont{T1}{ptm}{m}{n}
\rput(11.565625,-0.47484374){\scriptsize $G_5$}
\usefont{T1}{ptm}{m}{n}
\rput(13.825625,0.40515625){\scriptsize $G_4$}
\usefont{T1}{ptm}{m}{n}
\rput(12.205625,0.80515623){\scriptsize $G_3$}
\usefont{T1}{ptm}{m}{n}
\rput(12.705625,0.30515626){\scriptsize $G_2$}
\usefont{T1}{ptm}{m}{n}
\rput(12.245625,-0.13484375){\scriptsize $G_1$}
\usefont{T1}{ptm}{m}{n}
\rput(11.784062,1.0651562){\tiny -8}
\usefont{T1}{ptm}{m}{n}
\rput(13.06625,-0.23484375){\tiny -6}
\usefont{T1}{ptm}{m}{n}
\rput(12.243281,-1.7148438){\tiny -5}
\usefont{T1}{ptm}{m}{n}
\rput(8.946719,-0.51484376){\tiny -4}
\usefont{T1}{ptm}{m}{n}
\rput(14.381406,-0.07484375){\tiny -3}
\usefont{T1}{ptm}{m}{n}
\rput(11.365937,1.4451562){\tiny -2}
\usefont{T1}{ptm}{m}{n}
\rput(9.765938,-1.1348437){\tiny -2}
\usefont{T1}{ptm}{m}{n}
\rput(12.185938,0.12515625){\tiny -2}
\usefont{T1}{ptm}{m}{n}
\rput(10.565937,0.16515625){\tiny -2}
\usefont{T1}{ptm}{m}{n}
\rput(9.685938,0.6251562){\tiny -2}
\usefont{T1}{ptm}{m}{n}
\rput(8.465938,-0.83484375){\tiny -2}
\usefont{T1}{ptm}{m}{n}
\rput(12.785937,-0.6148437){\tiny -2}
\usefont{T1}{ptm}{m}{n}
\rput(8.965938,0.20515625){\tiny -2}
\usefont{T1}{ptm}{m}{n}
\rput(12.325937,0.38515624){\tiny -2}
\usefont{T1}{ptm}{m}{n}
\rput(14.865937,0.06515625){\tiny -2}
\usefont{T1}{ptm}{m}{n}
\rput(10.665937,-0.6148437){\tiny -2}
\psline[linewidth=0.04cm](0.5776353,0.7589219)(2.2688322,0.7589219)
\psline[linewidth=0.04cm](0.83781946,0.8945383)(0.18735908,-0.326009)
\psline[linewidth=0.04cm](2.0086482,0.8945383)(2.6591084,-0.326009)
\psline[linewidth=0.04cm](0.18735908,-0.19039264)(0.83781946,-1.4109399)
\psline[linewidth=0.04cm](2.6591084,-0.19039264)(2.1387403,-1.4109399)
\psline[linewidth=0.04cm](0.5776353,-1.2753235)(2.3989244,-1.2753235)
\psbezier[linewidth=0.04](0.56,-0.69984376)(0.0,-1.2798438)(0.4424972,-1.7255973)(0.86,-1.7198437)(1.2775028,-1.7140903)(4.92,-1.8398438)(6.26,-1.5598438)
\psline[linewidth=0.04cm](6.158585,1.5861818)(6.158585,-1.668611)
\rput{-89.926735}(5.9457974,6.2682037){\psarc[linewidth=0.04](6.111011,0.1573991){0.560153}{-11.053917}{190.84953}}
\psbezier[linewidth=0.04](1.6964271,0.5826206)(1.748464,0.7860452)(1.5602465,1.3910898)(2.24,1.5001563)(2.9197536,1.6092228)(6.34,2.1801562)(6.6008983,1.4776887)(6.8617964,0.77522105)(6.672935,0.7297439)(6.48,0.42015624)
\psline[linewidth=0.04cm](2.28,-0.7998437)(6.3,-0.7998437)
\psbezier[linewidth=0.04](3.44,1.8401562)(3.26,-0.21984375)(3.98,-0.41984376)(3.92,0.12015625)(3.86,0.66015625)(3.3,-0.41984376)(3.36,-1.9398438)
\psbezier[linewidth=0.04](4.7,1.9601562)(4.54,-0.13984375)(5.26,-0.39984375)(5.2,0.14015625)(5.14,0.68015623)(4.58,-0.29984376)(4.66,-1.7998438)
\usefont{T1}{ptm}{m}{n}
\rput(15.185625,0.20515625){\scriptsize $L$}
\usefont{T1}{ptm}{m}{n}
\rput(10.925625,0.38515624){\scriptsize $L_3$}
\usefont{T1}{ptm}{m}{n}
\rput(9.825625,-1.4348438){\scriptsize $E_1$}
\usefont{T1}{ptm}{m}{n}
\rput(0.765625,0.26515624){\scriptsize $L_1$}
\usefont{T1}{ptm}{m}{n}
\rput(2.065625,0.24515624){\scriptsize $L_3$}
\usefont{T1}{ptm}{m}{n}
\rput(2.165625,-0.63484377){\scriptsize $L_2$}
\usefont{T1}{ptm}{m}{n}
\rput(6.765625,0.14515625){\scriptsize $L$}
\usefont{T1}{ptm}{m}{n}
\rput(5.995625,0.14515625){\scriptsize $C$}
\usefont{T1}{ptm}{m}{n}
\rput(3.865625,0.42515624){\scriptsize $F_1$}
\usefont{T1}{ptm}{m}{n}
\rput(5.125625,0.48515624){\scriptsize $F_2$}
\usefont{T1}{ptm}{m}{n}
\rput(1.445625,-1.0948437){\scriptsize $E_1$}
\usefont{T1}{ptm}{m}{n}
\rput(0.705625,-0.59484375){\scriptsize $E_2$}
\usefont{T1}{ptm}{m}{n}
\rput(1.845625,-1.8948437){\scriptsize $E_3$}
\usefont{T1}{ptm}{m}{n}
\rput(2.705625,1.3651563){\scriptsize $E_5$}
\usefont{T1}{ptm}{m}{n}
\rput(4.085625,-0.95484376){\scriptsize $E_7$}
\usefont{T1}{ptm}{m}{n}
\rput(7.661406,-0.10984375){$\longleftarrow$}
\usefont{T1}{ptm}{m}{n}
\rput(1.405625,0.54515624){\scriptsize $E_4$}
\end{pspicture}
}
\end{figure}

$E_3,G_{10}=[5,2]$, $G_2,G_1,F_1,E_5,E_4,L_1,E_2=[2,2,8,2,2,2,4]$,

$L_2,C,F_2;E_7=[2,3,6;2]$, $K_X^2=-11+2+7+4=2$,

\begin{eqnarray*}
K_Z &\equiv & \frac{-F_1}{2}+\frac{-F_2}{2}+\frac{1}{2}G_2+G_3+\frac{1}{2}G_5+\frac{1}{2}G_6\\
 & & +\frac{1}{2}G_7+\frac{1}{2}G_8+G_9+G_{10}+2G_{11}
\end{eqnarray*}

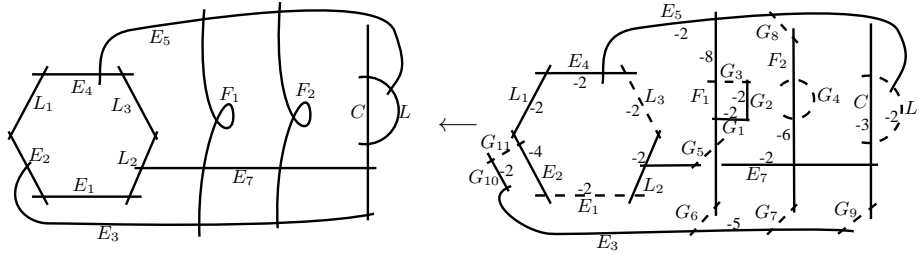
\begin{eqnarray*}
f^*(K_X) &\equiv & \frac{7}{16}F_1+\frac{2}{6}F_2+\frac{10}{16}G_1+\frac{13}{16}G_2+G_3+\frac{1}{2}G_5+\frac{1}{2}G_6\\
& & +\frac{1}{2}G_7+\frac{1}{2}G_8+G_9+\frac{4}{3}G_{10}+2G_{11}+\frac{11}{16}E_2+\frac{2}{3}E_3\\
& & +\frac{13}{16}E_4+\frac{14}{16}E_5+E_7+\frac{12}{16}L_1+\frac{1}{2}L_2+\frac{2}{3}C
\end{eqnarray*}

Therefore we have $\Q$-Gorenstein smoothable KSBA surfaces with one singularity $[2,3,6;2]$, $p_g=0$, and $K^2=2$. These surfaces are rational via \cite{Urz13}.

Using the curve $G_{11}$ we get that loops around $E_2$ and $G_{10}$ are homotopic and since the orders of the fundamental groups of the links are coprime we get that both loops are trivial. Then using $G_5,G_8$ and $G_9$ we get that loops around $E_7,F_2$ and $C$ are trivial. Finally using Mumford's relations we have that $1=\alpha_1 \alpha_2 \alpha_3 \gamma^{-2}$ where $\alpha_1,\alpha_2,\alpha_3$ and $\gamma$ are loops around $L_2,C,F_2$ and $E_7$ respectively. Since we have already shown that $\alpha_2,\alpha_3$ and $\gamma$ are trivial, we conclude that $\alpha_1$ is also trivial. Thus the fundamental group of the complement of the exceptional divisors in $Z$ is trivial.

\section{Overview of Wahl singularities in these moduli spaces} \label{s3}

In the following tables we list new and old Wahl singularities that appear on normal KSBA surfaces with no-local-to-global obstructions to deform for the invariants $\pi_1=1$, $p_g=0$, and $K^2=1,2,3,4$. This is a list with the known (to us) examples, we give one reference for each. We remark that $K^2=4$ is the maximum allowed $K^2$ for the Lee-Park type of construction (see e.g. \cite{PSU13}).

Each table shows:
\begin{itemize}
\item Value of $n$ and $a$ corresponding to $\frac{1}{n^2}(1,na-1)$.
\item The associated Hirzebruch-Jung continued fraction $$\frac{n^2}{na-1}=e_1 - \frac{1}{e_2 - \frac{1}{\ddots - \frac{1}{e_s}}}=[e_1,\ldots,e_s].$$
\item Type of the minimal model of the resolution of the KSBA surface with that one singularity (TMMR). Here we have the following types: rational (Rat), Dolgachev surface of type $(2,3)$ (Dol(2,3)), general type with $K^2=k$ (GenTypek). To know the type, we apply the explicit MMP in \cite{HTU13} (see \cite{Urz13}).
\item In the last column, we put one reference where it was constructed (there may be more references).
\end{itemize}

\bigskip
\textbf{Wahl singularities for $K^2=1$:}

\begin{longtable}{|c|c|c|c|}
\hline

$\mathbf{\binom{n}{a}}$ & {\bf Chain} & {\bf TMMR} & {\bf Reference} \\ [0.05cm]

\hline

$\binom{2}{1}$ & $[4]$ & Dol(2,3) & \cite{LP07} \\ [0.05cm]

$\binom{3}{1}$ & $[5,2]$ & Dol(2,3) & \cite{S13} \\ [0.05cm]

$\binom{3}{1}$ & $[5,2]$ & Rat & \cite{PN} \\ [0.05cm]

$\binom{4}{1}$ & $[6,2,2]$ & Rat & \cite{Urz13} \\ [0.05cm]

$\binom{5}{1}$ & $[7,2,2,2]$ & Rat & \cite{Urz13} \\ [0.05cm]

$\binom{5}{2}$ & $[3,5,2]$ & Rat & \cite{Urz13} \\ [0.05cm]

$\binom{6}{1}$ & $[8,2,2,2,2]$ & Rat & \cite{LP07} \\ [0.05cm]

$\binom{7}{1}$ & $[9,2,2,2,2,2]$ & Rat & \cite{PSlist} \\ [0.05cm]

$\binom{7}{2}$ & $[4,5,2,2]$ & Rat & \cite{PN} \\ [0.05cm]

$\binom{7}{3}$ & $[2,6,2,3]$ & Rat & \cite{PSU13} \\ [0.05cm]

$\binom{8}{3}$ & $[3,5,3,2]$ & Rat & \cite{PSU13} \\ [0.05cm]

$\binom{9}{2}$ & $[5,5,2,2,2]$ & Rat & \cite{PN} \\ [0.05cm]

$\binom{9}{4}$ & $[2,7,2,2,3]$ & Rat & \cite{PSlist} \\ [0.05cm]

$\binom{10}{3}$ & $[2,2,6,2,4]$ & Rat & \cite{S13} \\ [0.05cm]

$\binom{11}{2}$ & $[6,5,2,2,2,2]$ & Rat & \cite{S13} \\ [0.05cm]

$\binom{11}{3}$ & $[4,5,3,2,2]$ & Rat & \cite{LN12} \\ [0.05cm]

$\binom{11}{4}$ & $[3,6,2,3,2]$ & Rat & \cite{S13} \\ [0.05cm]

$\binom{11}{5}$ & $[2,8,2,2,2,3]$ & Rat & \cite{PSlist} \\ [0.05cm]

$\binom{12}{5}$ & $[2,4,5,2,3]$ & Rat & \cite{S13} \\ [0.05cm]

$\binom{13}{3}$ & $[2,2,2,6,2,5]$ & Rat & \cite{PN} \\ [0.05cm]

$\binom{13}{4}$ & $[2,2,7,2,2,4]$ & Rat & \cite{S13} \\ [0.05cm]

$\binom{13}{5}$ & $[3,3,5,3,2]$ & Rat & \cite{S13} \\ [0.05cm]

$\binom{14}{3}$ & $[5,5,3,2,2,2]$ & Rat & \cite{S13} \\ [0.05cm]

$\binom{16}{5}$ & $[2,2,8,2,2,2,4]$ & Rat & \cite{PN} \\ [0.05cm]

$\binom{16}{7}$ & $[2,5,5,2,2,3]$ & Rat & \cite{S13} \\ [0.05cm]

$\binom{17}{5}$ & $[4,2,5,4,2,2]$ & Rat & \cite{PN} \\ [0.05cm]

$\binom{17}{7}$ & $[2,4,2,6,2,3]$ & Rat & \cite{S13} \\ [0.05cm]

$\binom{18}{5}$ & $[4,3,5,3,2,2]$ & Rat & \cite{S13} \\ [0.05cm]

$\binom{19}{8}$ & $[2,5,5,2,2,3]$ & Rat & \cite{S13} \\ [0.05cm]

$\binom{21}{5}$ & $[2,2,2,8,2,2,2,5]$ & Rat & \cite{S13} \\ [0.05cm]

$\binom{23}{10}$ & $[3,2,2,6,2,5,2]$ & Rat & \cite{S13} \\ [0.05cm]

$\binom{24}{5}$ & $[5,7,2,2,3,2,2,2]$ & Rat & \cite{PSU13} \\ [0.05cm]

\hline

\end{longtable}


Note that the singularity in $\binom{3}{1}$ appears twice in the table above but in two different ways. This is the only singularity known so far in these moduli spaces with this property. This situation is opposite to the one for $\frac{1}{4}(1,1)$, where the TMMR must be a Dol(2,3) (see e.g. \cite{Urz13}). 

\bigskip
\textbf{Wahl singularities for $K^2=2$:}

\begin{longtable}{|c|c|c|c|}

\hline

$\mathbf{\binom{n}{a}}$ & {\bf Chain} & {\bf TMMR} & {\bf Reference}\\ [0.05cm]

\hline

$\binom{2}{1}$ & $[4]$ & GenType1 & \cite{LP07} \\ [0.05cm]

$\binom{3}{1}$ & $[5,2]$ & Dol(2,3) & \cite{LP07} \\ [0.05cm]

$\binom{4}{1}$ & $[6,2,2]$ & Dol(2,3) & \cite{PPS09} \\ [0.05cm]

$\binom{5}{1}$ & $[7,2,2,2]$ & Rat & \cite{LP07} \\ [0.05cm]

$\binom{5}{2}$ & $[3,5,2]$ & Dol(2,3) & \cite{PSlist} \\ [0.05cm]

$\binom{6}{1}$ & $[8,2,2,2,2]$ & Rat & \cite{PSlist} \\ [0.05cm]

$\binom{7}{1}$ & $[9,2,2,2,2,2]$ & Rat & \cite{S13} \\ [0.05cm]

$\binom{7}{2}$ & $[4,5,2,2]$ & Rat & \cite{PSlist} \\ [0.05cm]

$\binom{7}{3}$ & $[2,6,2,3]$ & Rat & \cite{S13} \\ [0.05cm]

$\binom{8}{3}$ & $[3,5,3,2]$ & Rat & \cite{S13} \\ [0.05cm]

$\binom{9}{4}$ & $[2,7,2,2,3]$ & Rat & \cite{LP07} \\ [0.05cm]

$\binom{10}{1}$ & $[12,2,2,2,2,2,2,2,2]$ & Rat & \cite{PN} \\ [0.05cm]

$\binom{10}{3}$ & $[4,2,6,2,2]$ & Rat & \cite{S13} \\ [0.05cm]

$\binom{11}{2}$ & $[6,5,2,2,2,2]$ & Rat & \cite{S13} \\ [0.05cm]

$\binom{11}{3}$ & $[4,5,3,2,2]$ & Rat & \cite{S13} \\ [0.05cm]

$\binom{11}{4}$ & $[3,6,2,3,2]$ & Rat & \cite{PSlist} \\ [0.05cm]

$\binom{11}{5}$ & $[2,8,2,2,2,3]$ & Rat & \cite{PN} \\ [0.05cm]

$\binom{12}{5}$ & $[2,4,5,2,3]$ & Rat & \cite{S13} \\ [0.05cm]

$\binom{13}{3}$ & $[5,2,6,2,2,2]$ & Rat & \cite{S13} \\ [0.05cm]

$\binom{13}{4}$ & $[2,2,7,2,2,4]$ & Rat & \cite{PSlist} \\ [0.05cm]

$\binom{13}{5}$ & $[3,3,5,3,2]$ & Rat & \cite{S13} \\ [0.05cm]

$\binom{13}{6}$ & $[2,9,2,2,2,2,3]$ & Rat & \cite{PN} \\ [0.05cm]

$\binom{14}{3}$ & $[5,5,3,2,2,2]$ & Rat & \cite{S13} \\ [0.05cm]

$\binom{14}{5}$ & $[2,4,5,2,3]$ & Rat & \cite{S13} \\ [0.05cm]

$\binom{15}{4}$ & $[4,6,2,3,2,2]$ & Rat & \cite{S13} \\ [0.05cm]

$\binom{15}{7}$ & $[2,10,2,2,2,2,2,3]$ & Rat & \cite{LP07} \\ [0.05cm]

$\binom{16}{5}$ & $[2,2,8,2,2,2,4]$ & Rat & \cite{PSlist} \\ [0.05cm]

$\binom{17}{4}$ & $[2,2,2,7,2,2,5]$ & Rat & \cite{S13} \\ [0.05cm]

$\binom{17}{6}$ & $[3,8,2,2,2,3,2]$ & Rat & \cite{S13} \\ [0.05cm]

$\binom{17}{7}$ & $[2,4,2,6,2,3]$ & Rat & \cite{S13} \\ [0.05cm]

$\binom{17}{8}$ & $[2,11,2,2,2,2,2,2,3]$ & Rat & \cite{PSlist} \\ [0.05cm]

$\binom{18}{7}$ & $[2,3,6,2,3,3]$ & Rat & \cite{S13} \\ [0.05cm]

$\binom{19}{5}$ & $[4,7,2,2,3,2,2]$ & Rat & \cite{PSlist} \\ [0.05cm]

$\binom{19}{6}$ & $[2,2,9,2,2,2,2,4]$ & Rat & \cite{PSlist} \\ [0.05cm]

$\binom{19}{7}$ & $[3,4,5,2,3,2]$ & Rat & \cite{S13} \\ [0.05cm]

$\binom{19}{8}$ & $[2,4,5,3,2,3]$ & Rat & \cite{S13} \\ [0.05cm]

$\binom{20}{9}$ & $[2,6,5,2,2,2,3]$ & Rat & \cite{S13} \\ [0.05cm]

$\binom{21}{8}$ & $[3,3,5,3,3,2]$ & Rat & \cite{S13} \\ [0.05cm]

$\binom{22}{7}$ & $[2,2,10,2,2,2,2,2,4]$ & Rat & \cite{PSlist} \\ [0.05cm]

$\binom{23}{4}$ & $[6,6,2,3,2,2,2,2]$ & Rat & \cite{LN12} \\ [0.05cm]

$\binom{23}{8}$ & $[3,10,2,2,2,2,2,3,2]$ & Rat & \cite{PSlist} \\ [0.05cm]

$\binom{24}{5}$ & $[5,7,2,2,3,2,2,2]$ & Rat & \cite{PSU13} \\ [0.05cm]

$\binom{24}{11}$ & $[2,7,5,2,2,2,2,3]$ & Rat & \cite{S13} \\ [0.05cm]

$\binom{25}{8}$ & $[2,2,11,2,2,2,2,2,2,4]$ & Rat & \cite{LP07} \\ [0.05cm]

$\binom{25}{9}$ & $[3,5,5,2,2,3,2]$ & Rat & \cite{S13} \\ [0.05cm]

$\binom{27}{10}$ & $[3,4,2,6,2,3,2]$ & Rat & \cite{S13} \\ [0.05cm]

$\binom{29}{4}$ & $[8,2,2,7,2,2,2,2,2,2]$ & Rat & \cite{S13} \\ [0.05cm]

$\binom{29}{13}$ & $[2,6,2,6,2,2,2,3]$ & Rat & \cite{S13} \\ [0.05cm]

$\binom{30}{11}$ & $[3,4,5,3,2,3,2]$ & Rat & \cite{Urz13} \\ [0.05cm]

$\binom{31}{13}$ & $[2,4,3,5,3,2,3]$ & Rat & \cite{S13} \\ [0.05cm]

$\binom{34}{9}$ & $[4,5,5,2,2,3,2,2]$ & Rat & \cite{S13} \\ [0.05cm]

$\binom{35}{11}$ & $[2,2,7,5,2,2,2,2,4]$ & Rat & \cite{S13} \\ [0.05cm]

$\binom{37}{17}$ & $[2,7,5,3,2,2,2,2,3]$ & Rat & \cite{PN} \\ [0.05cm]

$\binom{41}{19}$ & $[2,8,2,6,2,2,2,2,2,3]$ & Rat & \cite{S13} \\ [0.05cm]

$\binom{42}{13}$ & $[2,2,6,2,6,2,2,2,4]$ & Rat & \cite{S13} \\ [0.05cm]

$\binom{42}{19}$ & $[2,6,6,2,3,2,2,2,3]$ & Rat & \cite{S13} \\ [0.05cm]

$\binom{44}{17}$ & $[2,3,4,2,6,2,3,3]$ & Rat & \cite{S13} \\ [0.05cm]

$\binom{46}{21}$ & $[2,7,2,2,7,2,2,2,2,3]$ & Rat & \cite{S13} \\ [0.05cm]

$\binom{49}{8}$ & $[2,2,2,2,2,11,2,2,2,2,2,2,7]$ & Rat & \cite{S13} \\ [0.05cm]

$\binom{55}{9}$ & $[2,2,2,2,2,12,2,2,2,2,2,2,2,7]$ & Rat & \cite{S13} \\ [0.05cm]

$\binom{58}{13}$ & $[5,2,9,2,2,2,2,4,2,2,2]$ & Rat & \cite{S13} \\ [0.05cm]

\hline

\end{longtable}


For this case, the largest index is $n=58$. The chain corresponding to the exceptional divisor of the minimal resolution of this singularity was obtained by blowing up nine times the elliptic fibration $I_4+8I_1$ with one double section as shown in Figure \ref{k2}.

\begin{figure}[htbp]
\scalebox{1} 
{
\begin{pspicture}(0,-2.0705879)(7.445918,2.0905879)
\psline[linewidth=0.04cm](0.24721007,1.4967585)(0.24721007,-1.5222596)
\psline[linewidth=0.04cm](0.0,1.2451737)(2.7193108,1.2451737)
\psline[linewidth=0.04cm,linestyle=dashed,dash=0.16cm 0.16cm](2.4721007,1.4967585)(2.4721007,-1.5222596)
\psline[linewidth=0.04cm](0.0,-1.2706747)(2.7193108,-1.2706747)
\psbezier[linewidth=0.04](1.8046335,1.0439057)(1.7304705,1.6728679)(2.0024016,1.7231848)(2.6204267,1.7483433)(3.238452,1.7735019)(5.488064,1.7231848)(6.7735558,1.7483433)
\usefont{T1}{ptm}{m}{n}
\rput(0.47557616,0.5746112){\scriptsize $L_3$}
\usefont{T1}{ptm}{m}{n}
\rput(2.9555762,1.9346112){\scriptsize $E_2$}
\usefont{T1}{ptm}{m}{n}
\rput(6.205576,-0.36538878){\scriptsize $F_2$}
\usefont{T1}{ptm}{m}{n}
\rput(1.8555762,-0.4453888){\scriptsize $C$}
\usefont{T1}{ptm}{m}{n}
\rput(3.525576,-1.6853888){\scriptsize $S$}
\usefont{T1}{ptm}{m}{n}
\rput(1.1155761,1.4346112){\scriptsize $L_2$}
\usefont{T1}{ptm}{m}{n}
\rput(1.2755761,-1.4053888){\scriptsize $L_1$}
\usefont{T1}{ptm}{m}{n}
\rput(4.445576,-1.0653888){\scriptsize $F_1$}
\usefont{T1}{ptm}{m}{n}
\rput(2.7155762,0.5946112){\scriptsize $E_1$}
\psline[linewidth=0.04cm](4.2025714,1.2451737)(4.227292,-2.050588)
\psline[linewidth=0.04cm](6.4274616,1.9496112)(6.476904,-2.050588)
\rput{-90.0947}(3.9708898,3.802475){\psarc[linewidth=0.04,linestyle=dashed,dash=0.16cm 0.16cm](3.8835423,-0.08092839){0.6498386}{13.378786}{180.79324}}
\rput{-90.0947}(6.2786493,6.0006185){\psarc[linewidth=0.04,linestyle=dashed,dash=0.16cm 0.16cm](6.134679,-0.13383092){0.65304023}{13.552043}{179.15216}}
\psbezier[linewidth=0.04](2.2990537,-1.0442485)(2.5215428,-0.9436145)(2.867637,-0.91845596)(3.2137308,-1.1951994)(3.559825,-1.4719427)(3.856477,-1.7990029)(4.227292,-1.8996369)(4.5981073,-2.0002708)(6.3532987,-1.9751123)(6.649951,-1.9499538)
\psbezier[linewidth=0.04](1.8293545,-1.4467841)(1.6563075,-0.7675051)(2.0631902,-0.4238442)(2.5957057,-0.23917691)(3.1282213,-0.05450963)(3.485662,0.037566412)(4.1036873,0.01240793)
\psline[linewidth=0.04cm,linestyle=dashed,dash=0.16cm 0.16cm](3.590383,1.8541358)(4.38,0.8696112)
\psline[linewidth=0.04cm,linestyle=dashed,dash=0.16cm 0.16cm](3.086521,0.10788338)(3.12,-1.3503888)
\psline[linewidth=0.04cm](5.5775056,1.1393813)(6.2402625,1.1405945)
\psline[linewidth=0.04cm](5.6920295,0.26397294)(5.655449,1.1944323)
\psline[linewidth=0.04cm](5.8142962,0.37340757)(5.0232906,0.36296707)
\psline[linewidth=0.04cm,linestyle=dashed,dash=0.16cm 0.16cm](6.08,1.2496113)(6.56,0.7296112)
\psbezier[linewidth=0.04](4.326176,-0.012750555)(4.9194803,-0.088226005)(4.5981073,0.7168455)(4.8700385,0.6665285)(5.1419697,0.61621153)(5.279419,-0.37816626)(5.3644586,-0.7926636)(5.449498,-1.2071608)(5.8341575,-0.4404448)(5.7847157,-1.3713087)
\psline[linewidth=0.04cm,linestyle=dashed,dash=0.16cm 0.16cm](5.5869474,-1.2455163)(6.674672,-1.2455163)
\usefont{T1}{ptm}{m}{n}
\rput(3.3455763,-0.5653888){\scriptsize $G_9$}
\usefont{T1}{ptm}{m}{n}
\rput(6.305815,0.15461121){\tiny -9}
\usefont{T1}{ptm}{m}{n}
\rput(4.083291,-0.3853888){\tiny -5}
\usefont{T1}{ptm}{m}{n}
\rput(3.5067284,0.1346112){\tiny -4}
\usefont{T1}{ptm}{m}{n}
\rput(4.7659764,1.5946112){\tiny -2}
\usefont{T1}{ptm}{m}{n}
\rput(3.6255763,1.3946112){\scriptsize $G_8$}
\usefont{T1}{ptm}{m}{n}
\rput(6.125576,-1.3653888){\scriptsize $G_7$}
\usefont{T1}{ptm}{m}{n}
\rput(6.7855763,0.7946112){\scriptsize $G_6$}
\usefont{T1}{ptm}{m}{n}
\rput(5.8255763,1.3146112){\scriptsize $G_5$}
\usefont{T1}{ptm}{m}{n}
\rput(5.405576,0.83461124){\scriptsize $G_4$}
\usefont{T1}{ptm}{m}{n}
\rput(5.3055763,0.5546112){\scriptsize $G_3$}
\usefont{T1}{ptm}{m}{n}
\rput(7.025576,-0.045388788){\scriptsize $G_2$}
\usefont{T1}{ptm}{m}{n}
\rput(4.705576,-0.3853888){\scriptsize $G_1$}
\usefont{T1}{ptm}{m}{n}
\rput(5.7859764,0.77461123){\tiny -2}
\usefont{T1}{ptm}{m}{n}
\rput(0.38597655,0.1346112){\tiny -2}
\usefont{T1}{ptm}{m}{n}
\rput(5.4659767,0.23461121){\tiny -2}
\usefont{T1}{ptm}{m}{n}
\rput(5.1259766,-1.8453888){\tiny -2}
\usefont{T1}{ptm}{m}{n}
\rput(5.9659767,1.0146112){\tiny -2}
\usefont{T1}{ptm}{m}{n}
\rput(2.3059766,0.1946112){\tiny -2}
\usefont{T1}{ptm}{m}{n}
\rput(1.3059765,1.0746112){\tiny -2}
\usefont{T1}{ptm}{m}{n}
\rput(1.1459765,-1.1653888){\tiny -2}
\end{pspicture}
}
\caption{Largest index for $K^2=2$}
\label{k2}
\end{figure}
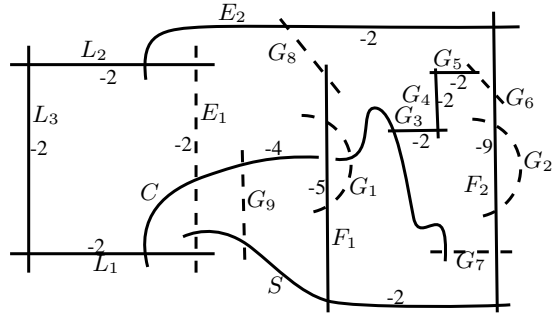

\begin{figure}[htbp]
\scalebox{0.75} 
{
\begin{pspicture}(0,-3.9101562)(9.17125,3.8901563)
\psline[linewidth=0.04cm](1.2053125,3.8701563)(0.2053125,2.1901562)
\psline[linewidth=0.04cm](1.6053125,3.2701561)(0.6053125,1.5901562)
\psline[linewidth=0.04cm](1.2053125,1.4701562)(0.2053125,-0.20984375)
\psline[linewidth=0.04cm](1.1718014,0.9710939)(0.23882361,2.6892185)
\psline[linewidth=0.04cm](1.2053125,-0.9298437)(0.2053125,-2.6098437)
\psline[linewidth=0.04cm](1.1718014,-1.4289061)(0.23882361,0.2892186)
\psline[linewidth=0.04cm](1.1253124,-3.6898437)(0.23882361,-2.1107814)
\psline[linewidth=0.04cm](1.5718014,-3.028906)(0.6388236,-1.3107814)
\psline[linewidth=0.04cm](0.8053125,-3.5498438)(8.465313,-3.5898438)
\usefont{T1}{ptm}{m}{n}
\rput(2.3009374,3.4551563){\scriptsize $L'$}
\usefont{T1}{ptm}{m}{n}
\rput(1.5909375,2.0751562){\scriptsize $L''$}
\usefont{T1}{ptm}{m}{n}
\rput(0.6309375,3.4151564){\scriptsize $E_1$}
\usefont{T1}{ptm}{m}{n}
\rput(0.3309375,1.9351562){\scriptsize $E_2$}
\usefont{T1}{ptm}{m}{n}
\rput(0.5309375,0.85515624){\scriptsize $E_3$}
\usefont{T1}{ptm}{m}{n}
\rput(0.3309375,-0.5648438){\scriptsize $E_4$}
\usefont{T1}{ptm}{m}{n}
\rput(0.4309375,-3.1248438){\scriptsize $E_6$}
\usefont{T1}{ptm}{m}{n}
\rput(2.4909375,-3.7248437){\scriptsize $E_7$}
\usefont{T1}{ptm}{m}{n}
\rput(0.4109375,-1.7648437){\scriptsize $E_5$}
\usefont{T1}{ptm}{m}{n}
\rput(7.9509373,0.07515625){\scriptsize $F_1$}
\usefont{T1}{ptm}{m}{n}
\rput(3.7709374,-0.18484375){\scriptsize $F_2$}
\psline[linewidth=0.04cm](4.0253124,2.4101562)(4.0653124,-3.0698438)
\psline[linewidth=0.04cm,linestyle=dashed,dash=0.16cm 0.16cm](4.1453123,-2.8498437)(3.3853126,-3.6698437)
\psline[linewidth=0.04cm](3.8253126,2.1101563)(4.8853126,2.1301563)
\psline[linewidth=0.04cm](4.7253127,2.9101562)(4.7453127,1.9901563)
\psline[linewidth=0.04cm](4.5653124,2.7901564)(5.6253123,2.8101563)
\psline[linewidth=0.04cm,linestyle=dashed,dash=0.16cm 0.16cm](5.5253124,3.5501564)(5.5453124,2.6901562)
\rput{-89.475075}(4.3763485,3.4329839){\psarc[linewidth=0.04,linestyle=dashed,dash=0.16cm 0.16cm](3.9204648,-0.49182224){0.6365934}{0.0}{180.0}}
\psline[linewidth=0.04cm](8.185312,2.6701562)(8.245313,-3.8098438)
\psline[linewidth=0.04cm](7.5453124,-1.7298437)(7.5453124,-2.5098438)
\psline[linewidth=0.04cm](7.6453123,-1.8098438)(6.9453125,-1.8098438)
\psline[linewidth=0.04cm](7.0253124,-1.1098437)(7.0453124,-1.8898437)
\psline[linewidth=0.04cm](6.4453125,-1.2098438)(7.1053123,-1.2098438)
\psline[linewidth=0.04cm](6.5253124,-0.58984375)(6.5253124,-1.2898438)
\psline[linewidth=0.04cm,linestyle=dashed,dash=0.16cm 0.16cm](7.4653125,-2.4098437)(8.425312,-2.4098437)
\psbezier[linewidth=0.04](0.9453125,3.7501562)(1.3853126,3.5301561)(1.9453125,3.7901564)(2.1853125,3.3101563)(2.4253125,2.8301563)(0.9359637,-3.4126143)(1.6253124,-3.1498437)(2.3146613,-2.887073)(1.8053125,-0.64984375)(2.5253124,-0.58984375)(3.2453125,-0.52984375)(3.4853125,-0.56984377)(3.8253126,-0.50984377)
\psbezier[linewidth=0.04](4.2253127,-0.52984375)(4.9853125,-0.50984377)(6.3453126,-0.16984375)(6.5853124,-0.46984375)(6.8253126,-0.76984376)(5.813494,-0.6903501)(5.7853127,-1.5898438)(5.757131,-2.4893374)(7.8253126,-2.7498438)(7.7053127,-3.3298438)
\psline[linewidth=0.04cm,linestyle=dashed,dash=0.16cm 0.16cm](7.6253123,-3.2298439)(8.345312,-3.2298439)
\psline[linewidth=0.04cm,linestyle=dashed,dash=0.16cm 0.16cm](8.045313,0.97015625)(9.045313,-0.00984375)
\psline[linewidth=0.04cm](9.005313,0.13015625)(8.125313,-1.0898438)
\psbezier[linewidth=0.04](1.1453125,2.9301562)(1.4853125,3.0101562)(1.7253125,2.6901562)(1.8453125,2.2901564)(1.9653125,1.8901563)(0.60538304,-3.4239764)(1.4053125,-3.3698437)(2.205242,-3.3157113)(3.0642757,-2.5176346)(3.1053126,-2.0698438)(3.1463494,-1.6220529)(3.1253126,-1.0098437)(3.1053126,-0.82984376)
\psbezier[linewidth=0.04](3.1053126,-0.38984376)(3.0653124,-0.02984375)(3.6253126,1.3101562)(3.4053125,2.4101562)(3.1853125,3.5101562)(5.7453127,3.8101563)(6.2453127,3.1901562)(6.7453127,2.5701563)(6.5948935,1.5535659)(6.7253127,0.55015624)(6.8557315,-0.4532534)(7.6453123,-0.28984374)(7.9853125,-0.32984376)
\psline[linewidth=0.04cm](8.345312,-0.30984375)(9.105312,-0.32984376)
\psline[linewidth=0.04cm,linestyle=dashed,dash=0.16cm 0.16cm](3.3253126,1.3101562)(4.2853127,1.3301562)
\psline[linewidth=0.04cm,linestyle=dashed,dash=0.16cm 0.16cm](2.1053126,-0.9898437)(3.2853124,-0.9898437)
\usefont{T1}{ptm}{m}{n}
\rput(2.7209375,-1.1448437){\scriptsize $G_{17}$}
\usefont{T1}{ptm}{m}{n}
\rput(3.4209375,-3.1248438){\scriptsize $G_{16}$}
\usefont{T1}{ptm}{m}{n}
\rput(3.8009374,1.1551563){\scriptsize $G_{15}$}
\usefont{T1}{ptm}{m}{n}
\rput(5.8409376,3.1351562){\scriptsize $G_{14}$}
\usefont{T1}{ptm}{m}{n}
\rput(5.1609373,2.6551561){\scriptsize $G_{13}$}
\usefont{T1}{ptm}{m}{n}
\rput(5.0809374,2.3151562){\scriptsize $G_{12}$}
\usefont{T1}{ptm}{m}{n}
\rput(4.4409375,1.9551562){\scriptsize $G_{11}$}
\usefont{T1}{ptm}{m}{n}
\rput(7.9209375,-3.0648437){\scriptsize $G_{10}$}
\usefont{T1}{ptm}{m}{n}
\rput(7.8509374,-2.5648437){\scriptsize $G_9$}
\usefont{T1}{ptm}{m}{n}
\rput(7.7709374,-2.0848436){\scriptsize $G_8$}
\usefont{T1}{ptm}{m}{n}
\rput(7.3509374,-1.6248437){\scriptsize $G_7$}
\usefont{T1}{ptm}{m}{n}
\rput(7.2709374,-1.3648437){\scriptsize $G_6$}
\usefont{T1}{ptm}{m}{n}
\rput(6.7909374,-1.0648438){\scriptsize $G_5$}
\usefont{T1}{ptm}{m}{n}
\rput(6.7309375,-0.84484375){\scriptsize $G_4$}
\usefont{T1}{ptm}{m}{n}
\rput(4.7309375,-0.20484374){\scriptsize $G_3$}
\usefont{T1}{ptm}{m}{n}
\rput(8.730938,0.61515623){\scriptsize $G_2$}
\usefont{T1}{ptm}{m}{n}
\rput(8.750937,-0.6248438){\scriptsize $G_1$}
\usefont{T1}{ptm}{m}{n}
\rput(7.95125,1.6551563){\tiny -12}
\usefont{T1}{ptm}{m}{n}
\rput(3.7309375,2.8751562){\tiny -7}
\usefont{T1}{ptm}{m}{n}
\rput(5.2120314,-0.58484375){\tiny -4}
\usefont{T1}{ptm}{m}{n}
\rput(0.85125,2.9751563){\tiny -2}
\usefont{T1}{ptm}{m}{n}
\rput(4.2309375,-1.8048438){\tiny -7}
\usefont{T1}{ptm}{m}{n}
\rput(1.01125,1.6351563){\tiny -2}
\usefont{T1}{ptm}{m}{n}
\rput(0.85125,0.59515625){\tiny -2}
\usefont{T1}{ptm}{m}{n}
\rput(0.75125,-0.36484376){\tiny -2}
\usefont{T1}{ptm}{m}{n}
\rput(0.59125,-2.2248437){\tiny -2}
\usefont{T1}{ptm}{m}{n}
\rput(0.81125,-2.8248436){\tiny -2}
\usefont{T1}{ptm}{m}{n}
\rput(5.11125,2.9551563){\tiny -2}
\usefont{T1}{ptm}{m}{n}
\rput(4.55125,2.4751563){\tiny -2}
\usefont{T1}{ptm}{m}{n}
\rput(4.31125,2.2551563){\tiny -2}
\usefont{T1}{ptm}{m}{n}
\rput(4.97125,-3.4048438){\tiny -2}
\usefont{T1}{ptm}{m}{n}
\rput(6.33125,-1.0048437){\tiny -2}
\usefont{T1}{ptm}{m}{n}
\rput(6.75125,-1.3848437){\tiny -2}
\usefont{T1}{ptm}{m}{n}
\rput(6.85125,-1.5848438){\tiny -2}
\usefont{T1}{ptm}{m}{n}
\rput(7.25125,-1.9848437){\tiny -2}
\usefont{T1}{ptm}{m}{n}
\rput(7.39125,-2.1648438){\tiny -2}
\usefont{T1}{ptm}{m}{n}
\rput(8.63125,-0.06484375){\tiny -2}
\usefont{T1}{ptm}{m}{n}
\rput(0.83125,-2.0048437){\tiny -2}
\usefont{T1}{ptm}{m}{n}
\rput(0.91125,2.4751563){\tiny -2}
\end{pspicture}
}
\caption{Largest index for $K^2=3$}
\label{k3}
\end{figure}
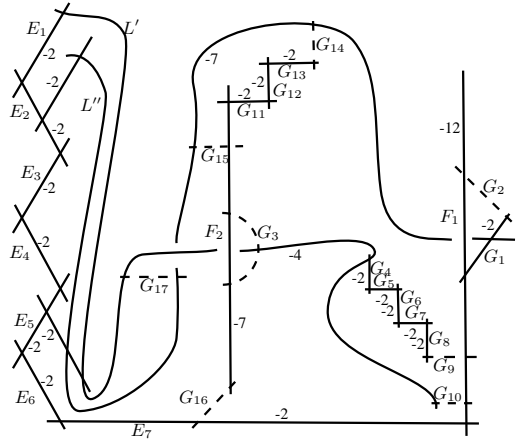

\textbf{Wahl singularities for $K^2=3$:}

\begin{longtable}{|c|c|c|c|}

\hline

$\mathbf{\binom{n}{a}}$ & {\bf Chain} & {\bf TMMR} & {\bf Reference}\\ [0.05cm]

\hline

$\binom{2}{1}$ & $[4]$ & GenType2 & \cite{PPS09} \\ [0.05cm]

$\binom{3}{1}$ & $[5,2]$ & GenType1 & \cite{Urz13} \\ [0.05cm]

$\binom{4}{1}$ & $[6,2,2]$ & GenType1 & \cite{PN} \\ [0.05cm]

$\binom{5}{1}$ & $[7,2,2,2]$ & Rat & \cite{S13} \\ [0.05cm]

$\binom{5}{2}$ & $[3,5,2]$ & GenType1 & \cite{S13} \\ [0.05cm]

$\binom{7}{1}$ & $[9,2,2,2,2,2]$ & Rat & \cite{PPS09} \\ [0.05cm]

$\binom{7}{3}$ & $[2,6,2,3]$ & Rat & \cite{PPS09} \\ [0.05cm]

$\binom{9}{4}$ & $[2,7,2,2,3]$ & Rat & \cite{S13} \\ [0.05cm]

$\binom{11}{4}$ & $[3,6,2,3,2]$ & Rat & \cite{S13} \\ [0.05cm]

$\binom{11}{5}$ & $[2,8,2,2,2,3]$ & Rat & \cite{S13} \\ [0.05cm]

$\binom{14}{3}$ & $[5,5,3,2,2,2]$ & Rat & \cite{PN} \\ [0.05cm]

$\binom{16}{5}$ & $[2,2,8,2,2,2,4]$ & Rat & \cite{Urz13} \\ [0.05cm]

$\binom{17}{7}$ & $[2,4,2,6,2,3]$ & Rat & \cite{S13} \\ [0.05cm]

$\binom{19}{5}$ & $[4,7,2,2,3,2,2]$ & Rat & \cite{PPS09} \\ [0.05cm]

$\binom{19}{6}$ & $[2,2,9,2,2,2,2,4]$ & Rat & \cite{S13} \\ [0.05cm]

$\binom{23}{6}$ & $[4,8,2,2,2,3,2,2]$ & Rat & \cite{PN} \\ [0.05cm]

$\binom{24}{7}$ & $[4,2,6,2,4,2,2]$ & Rat & \cite{S13} \\ [0.05cm]

$\binom{25}{7}$ & $[4,3,2,6,3,2,2]$ & Rat & \cite{S13} \\ [0.05cm]

$\binom{29}{7}$ & $[2,2,2,10,2,2,2,2,2,5]$ & Rat & \cite{S13} \\ [0.05cm]

$\binom{30}{11}$ & $[3,4,5,3,2,3,2]$ & Rat & \cite{Urz13} \\ [0.05cm]

$\binom{31}{7}$ & $[5,2,6,2,4,2,2,2]$ & Rat & \cite{S13} \\ [0.05cm]

$\binom{35}{6}$ & $[6,8,2,2,2,3,2,2,2,2]$ & Rat & \cite{PPS09} \\ [0.05cm]

$\binom{41}{17}$ & $[2,4,2,6,2,4,2,3]$ & Rat & \cite{S13} \\ [0.05cm]

$\binom{48}{17}$ & $[3,6,5,3,2,2,2,3,2]$ & Rat & \cite{PPS09} \\ [0.05cm]

$\binom{53}{14}$ & $[4,5,5,3,2,2,3,2,2]$ & Rat & \cite{PN} \\ [0.05cm]

$\binom{63}{29}$ & $[2,7,7,2,2,3,2,2,2,2,3]$ & Rat & \cite{LN12} \\ [0.05cm]

$\binom{65}{17}$ & $[4,6,5,3,2,2,2,3,2,2]$ & Rat & \cite{PPS09} \\ [0.05cm]

$\binom{69}{31}$ & $[2,6,2,6,2,4,2,2,2,3]$ & Rat & \cite{S13} \\ [0.05cm]

$\binom{71}{15}$ & $[5,4,6,2,3,2,3,2,2,2]$ & Rat & \cite{PN} \\ [0.05cm]

$\binom{83}{38}$ & $[2,7,2,6,2,4,2,2,2,2,3]$ & Rat & \cite{S13} \\ [0.05cm]

$\binom{97}{45}$ & $[2,8,2,6,2,4,2,2,2,2,2,3]$ & Rat & \cite{S13} \\ [0.05cm]

$\binom{100}{29}$ & $[4,2,6,2,6,2,2,2,4,2,2]$ & Rat & \cite{S13} \\ [0.05cm]

$\binom{100}{31}$ & $[4,2,2,2,4,2,6,2,6,2,2]$ & Rat & \cite{S13} \\ [0.05cm]

$\binom{113}{25}$ & $[2,2,2,10,2,6,2,2,2,2,2,2,2,5]$ & Rat & \cite{S13}\\ [0.05cm]

$\binom{113}{42}$ & $[2,3,2,6,2,6,2,2,2,4,3]$ & Rat & \cite{S13} \\ [0.05cm]

$\binom{123}{19}$ & $[7,2,12,2,2,2,2,2,2,2,4,2,2,2,2,2]$ & Rat & \cite{S13} \\ [0.05cm]

\hline

\end{longtable}

In this case, the largest index is $n=123$. The chain corresponding to the exceptional divisor of the minimal resolution of this singularity was obtained by blowing up seventeen times the elliptic fibration $I_3^*+3I_1$ with two double sections. This configuration has another chain corresponding to the Wahl singularity with $n=5$ and $a=1$, as shown in Figure \ref{k3}.

\bigskip
\textbf{Wahl singularities for $K^2=4$:}

\begin{longtable}{|c|c|c|c|}

\hline

$\mathbf{\binom{n}{a}}$ & {\bf Chain} & {\bf TMMR} & {\bf Reference} \\ [0.05cm]

\hline

$\binom{2}{1}$ & $[4]$ & GenType3 & \cite{PN} \\  [0.05cm]

$\binom{6}{1}$ & $[8,2,2,2,2]$ & Dol(2,3) & \cite{PN} \\  [0.05cm]

$\binom{7}{2}$ & $[4,5,2,2]$ & Dol(2,3) & \cite{HTU13} \\ [0.05cm]

$\binom{9}{4}$ & $[2,7,2,2,3]$ & GenType1 & \cite{PN} \\ [0.05cm]

$\binom{24}{5}$ & $[5,7,2,2,3,2,2,2]$ & Rat & \cite{PN} \\ [0.05cm]

$\binom{183}{38}$ & $[5,6,2,6,2,4,2,2,2,3,2,2,2]$ & Rat & \cite{PPS09-2} \\ [0.05cm]

$\binom{252}{107}$ & $[2,4,6,2,6,2,4,2,2,2,3,2,3]$ & Rat & \cite{PPS09-2} \\ [0.05cm]
\hline
\end{longtable}

We end with some remarks.


\textbf{(1)} We do not know if $\Q Eq$ appears on smoothable KSBA surfaces with $p_g=0$, and $K^2=3,4$.


\textbf{(2)} There exist a general bound for the indices of normal surface singularities that show up in the KSBA compactification. In particular, Y. Lee gives in \cite{L99} an explicit bound for the indices $n$ when the TMMR is of general type: $n \leq 2^{400 (K^2)^4}$. The largest indices we found above are not even close to this bound (and the majority has TMMR rational). We do not know an optimal bound.


\textbf{(3)} For $K^2=1,2,3$ we have many examples of Wahl singularities, for $K^2=4$ have very few. Essentially none of them is new in this case, they come from degenerations of the examples in \cite{PPS09-2}. The moduli space for $K^2=4$ is expected to be a surface, and so it may be simpler to study. With this in hand, it would be interesting to construct new $K^2=4$ examples, they would define KSBA boundary curves which may connect the known examples.



{\tiny Facultad de Matem\'aticas, Pontificia Universidad Cat\'olica de Chile, Santiago, Chile.}

\end{document}